\documentclass[a4paper,10pt,leqno]{amsart}
\usepackage{amssymb}
\usepackage{amsmath}
\usepackage{amsthm}
\usepackage{mathrsfs}
\usepackage{color}

\usepackage{graphicx}
\newtheorem{thm}{Theorem}[section]
\newtheorem{lem}[thm]{Lemma}

\theoremstyle{definition}
\newtheorem{rmk}{Remark}

\newcommand{\la}{\lambda}
\newcommand{\va}{\varphi}

\newcommand{\ww}{\widetilde{w}}

\newcommand{\wa}{\widetilde{a}}
\newcommand{\wb}{\widetilde{b}}
\newcommand{\wc}{\widetilde{c}}

\newcommand{\wm}{\widetilde{\mu}}
\newcommand{\wva}{\widetilde{\varphi}}
\newcommand{\wpsi}{\widetilde{\psi}}

\newcommand{\wf}{\widetilde{f}}
\newcommand{\wv}{\widetilde{v}}
\newcommand{\wy}{\widetilde{y}}

\newcommand{\cA}{\mathcal{A}}
\newcommand{\bb}{\mathbf{b}}
\newcommand{\bp}{\mathbf{p}}
\newcommand{\wF}{\widetilde{F}}
\newcommand{\wJ}{\widetilde{J}}

\DeclareMathOperator{\dd}{div}

\newcommand{\R}{\mathbb{R}}

\newcommand{\pp}{\partial}

\newcommand{\ba}{\begin{eqnarray*}}
\newcommand{\ea}{\end{eqnarray*}}
\newcommand{\be}{\begin{equation}}
\newcommand{\ee}{\end{equation}}
\newcommand{\bea}{\begin{eqnarray}}
\newcommand{\eea}{\end{eqnarray}}

\allowdisplaybreaks

\title[]{%
Lipschitz stability in inverse source and inverse coefficient problems 
for a first- and half-order time-fractional diffusion equation
}
\author{Atsushi Kawamoto}
\address{Department of Mathematical Sciences, The University of Tokyo, 
Komaba, Meguro, Tokyo 153-8914, Japan}
\email{kawamo@ms.u-tokyo.ac.jp}

\author{Manabu Machida}
\address{Institute for Medical Photonics Research, 
Hamamatsu University School of Medicine, Hamamatsu, Shizuoka 431-3192, Japan}
\email{machida@hama-med.ac.jp}

\date{}
\begin{document}
%The abstract of your paper
\begin{abstract}
We consider inverse problems for the first and half order time fractional equation.
We establish the stability estimates of Lipschitz type in inverse source and inverse coefficient problems by means of the Carleman estimates.
\end{abstract}
\maketitle

%\markright{\today}

%\,\hfill last update: \today, version: \jobname

\section{Introduction}

% Setting of Problems
%\noindent{\bf Setting of Problems}

Let $\Omega\subset \R^n$ be a bounded domain with the boundary $\pp\Omega$ of 
$C^2$ class. We set $Q=\Omega \times (0,T)$, where $T>0$. We use notations 
$\pp_t= \frac{\pp}{\pp t}$, $\pp_i=\frac{\pp}{\pp x_i}$ ($i=1,2,\ldots, n$). 
We also use the multi index $\alpha=(\alpha_1,\alpha_2, \ldots, \alpha_n)$ 
with $\alpha_j \in \mathbb{N}\cup\{0\}$ ($j=1,2,\ldots,n$), 
$\pp_x^\alpha =\pp_1^{\alpha_1}\pp_2^{\alpha_2} \cdots \pp_n^{\alpha_n}$, 
$|\alpha|=\alpha_1+\alpha_2+\cdots+\alpha_n$. Let $\nu =\nu (x)$ be the 
outwards unit normal vector to $\pp\Omega$ at $x$ and let 
$\pp_\nu =\nu \cdot \nabla$. In general, 
the $\beta$th order Caputo-type fractional derivative is defined by
\begin{equation*}
\pp_t^{\beta}u(x,t):=
\frac{1}{\Gamma\left(n-\beta\right)}\int_0^t\frac{1}{(t-\tau )^{\beta+1-n}}
\frac{\pp^nu(x,\tau)}{\pp\tau^n}\,d\tau,
\quad (x,t)\in Q,
\end{equation*}
for $n-1<\beta<n$, $n\in\mathbb{N}$ (See e.g., \cite{Caputo1967, Pod}). 
Here, $\Gamma$ is the gamma function. 
We consider the following first and half 
order time-fractional diffusion equation.% with the zero initial condition.
\begin{align}
\label{eq:1+1/2eq}
&
(\rho_1 \pp_t+\rho_2\pp_t^{\frac12}  - L)u(x,t)=g(x,t),
&
(x,t)\in Q,\\
\label{eq:b_condi}
&
u(x,t)=h_1(x,t),
&
(x,t)\in\pp\Omega\times(0,T),
\\
\label{eq:ini_condi}
&
u(x,0)=h_2(x),& x\in \Omega,
\end{align}
where $\rho_1>0$, $\rho_2 \neq 0$ are constants, and $L$ is a symmetric 
uniformly elliptic operator given by
\begin{equation}
%\label{eq:def_L2}
\nonumber
L u(x,t) :=\sum_{i,j=1}^n \pp_i (a_{ij}(x) \pp_j u(x,t))
-\sum_{j=1}^n b_j (x)\pp_j u(x,t)
- c(x)u(x,t),\ (x,t) \in Q.
\end{equation}
We assume that $a_{ij}\in C^3(\overline{\Omega})$, 
$a_{ij}=a_{ji}$ ($ i,j=1,2,\ldots, n$), 
$b_j \in  C^2(\overline{\Omega})$ ($j=1,2,\ldots, n$), 
$c \in C^2(\overline{\Omega})$, and moreover there exists a constant 
$\mu>0$ such that
\begin{equation}
%\label{eq:coe_a}
\nonumber
\frac1{\mu} |\xi|^2
\leq
\sum_{i,j=1}^n a_{ij}(x) \xi_i \xi_j
\leq
\mu |\xi|^2,
\quad
\xi=(\xi_1,\ldots, \xi_n) \in \R^n,\
x \in \overline{\Omega}.
\end{equation}

In fluid dynamics, (\ref{eq:1+1/2eq}) appears in the Basset problem 
\cite{Basset1910} when the motion of a particle in a nonuniform flow is 
considered \cite{Brennen05,Langlois-Farazmand-Haller15,Maxey-Riley83}. 
The first and half order time-fractional equation (\ref{eq:1+1/2eq}) also appears in porous media. 
Starting with the microscopic diffusion in a heterogeneous medium which has 
two length scales: the microscopic length scale of a typical porous block and 
the relative fracture width, a diffusion equation with the first- and 
half-order time derivatives is obtained at the large scale limit by the 
homogenization process \cite{Amaziane-etal04,APP}.

The first- and half-order equation (\ref{eq:1+1/2eq}) is one of parabolic 
equations with multiple time-fractional terms, i.e., the time-derivative part 
in the equation is given by $\sum_{j=0}^{\ell}p_j\pp_t^{\alpha_j}$, where 
$0<\alpha_{\ell}<\cdots<\alpha_1<\alpha_0\le 1$ and coefficients $p_j$ 
generally depend on $x$. Initial-boundary-value problems for multi-term 
time-fractional diffusion equations were considered in \cite{Luchko11}. 
In the case that all time-derivatives are non-integer order and the 
time-derivative part is given by $\sum_{j=1}^{\ell}p_j\pp_t^{\alpha_j}$, 
the well-posedness was investigated \cite{Li-Liu-Yamamoto15} and moreover 
the uniqueness in inverse boundary-value problems was proven 
\cite{Li-Imanuvilov-Yamamoto16}. An exact solution was obtained in the 
special case of a two-term time-fractional diffusion equation 
\cite{Bazhlekova-Dimovski14}. 
The uniqueness for two kinds of inverse problems of identifying fractional 
orders in diffusion equations with multiple time-fractional derivatives was 
proved \cite{Li-Yamamoto15}. The uniqueness in determining the spatial 
component of the source term from interiror observation was established 
\cite{Jiang-Li-Liu-Yamamoto17}. The maximum principle and uniqueness was considered in \cite{Liu17} for the determination of the temporal component of the source term from a single point observation. Also the unique continuation was considered for multi-term time-fractional diffusion equations \cite{Lin-Nakamura18}.

In \cite{Kwa}, the H\"{o}lder stability is proven for the inverse 
source problem of (\ref{eq:1+1/2eq}) (See also \cite{Li-Huang-Yamamoto}). 
In this paper, we further prove the 
Lipshitz stability not only for the inverse source problem but also for 
the inverse coefficient problem for (\ref{eq:1+1/2eq}).

%Related works to our inverse problems, 
%we may also refer to \cite{Li-Huang-Yamamoto} 
%in which they consider the H\"{o}lder stability in inverse source problem for the fractional diffusion equation with the time-derivative $\pp_t + \pp_t^\alpha$ ($0<\alpha\leq\frac34$) 
%under some additional assumptions on initial data. 

The methodology of our stability analysis is based on the technique of the 
Carleman estimate \cite{Carleman1939}, which was pioneered by 
Bukhgeim and Klibanov \cite{Bukhgeim-Klibanov81} when they proved the global 
uniqueness in inverse problems. See also \cite{Klibanov84,Klibanov92}, recent 
reviews \cite{Klibanov13,Yamamoto09}, and textbooks
\cite{Isakov06,Klibanov-Timonov04}. The Carleman estimate is a weighted $L^2$ 
inequality for a solution of a partial differential equation. In the case 
of parabolic equations with one first-order time derivative, the global 
Lipschitz stability was proven by using this method of Carleman estimates 
\cite{Imanuvilov-Yamamoto98}. In this paper we make use of the Carleman 
estimate for parabolic equations. 
The Carleman estimates have been used for 
differential equations with a single term time-fractional derivative 
\cite{Kawamoto18, Xu-Cheng-Yamamoto11,Yamamoto-Zhang12}.

This paper is organized as follows. In \S\ref{stability}, inverse source 
problems are considered. In \S\ref{porous}, inverse coefficient problems 
are considered. In \S\ref{carleman}, the Carleman estimate necessary for 
our paper is established. Finally, proofs of the main theorems are given 
in \S\ref{proof}.

\section{Inverse source problems}
\label{stability}

We consider the inverse problems of determining the time-independent source factor of \eqref{eq:1+1/2eq}
from spatial data and two types of observations. 
One is the boundary observation and the other is the interior observation. 

Let $t_0\in (0,T)$ be an arbitrarily fixed time.
Let $\gamma$ be an arbitrarily fixed open connected sub-boundary of $\pp\Omega$ 
and let $\omega$ be an arbitrary fixed sub-domain of $\Omega$ such that $\omega\Subset\Omega$. 
We set $\Sigma=\gamma \times (0,T)$ and $Q_\omega=\omega\times (0,T)$. 

Moreover we choose $\delta>0$ such that 
\begin{equation*}
0<t_0-\delta <t_0 <t_0+\delta <T,
\end{equation*}
and we set $Q_\delta =\Omega\times (t_0-\delta,t_0+\delta)$, $\Sigma_\delta=\gamma\times (t_0-\delta,t_0+\delta)$, 
$Q_{\omega,\delta}=\omega \times  (t_0-\delta,t_0+\delta)$. 

We assume that
\begin{equation}
\label{eq:R}
\left\{
\begin{aligned}
&R\in
C([0,T);C (\overline{\Omega}))\cap
C^2((0,T)%[0,T]
;C^2(\overline{\Omega}))\cap C^3((0,T)%[0,T]
;C(\overline{\Omega})), \\
&\pp_t^{\frac12} R \in C^2((0,T)%[0,T]
;C(\overline{\Omega})) %,\\
\text{ and}\ |R(x,t_0)|>0,\  x\in\overline{\Omega}.
\end{aligned}
\right.
\end{equation}
Furthermore we define
\begin{equation*}
\mathcal{U}=L^2(0,T;H^4(\Omega))\cap H^1(0,T;H^2(\Omega))\cap H^2(0,T;L^2(\Omega)).
\end{equation*}
Let us assume that $g(x,t)$ in (\ref{eq:1+1/2eq}) has the form
\begin{equation}
\label{eq:gfR}
g(x,t)=f(x)R(x,t),
\end{equation}
and set $h_1=0$ in $Q$, $h_2=0$ in $\Omega$.

We consider
\begin{align}
\label{eq:eq01}
&
(\rho_1\pp_t +\rho_2\pp_t^{\frac12} - L)u(x,t)=f(x)R(x,t),&
(x,t)\in Q,
\\
\label{eq:eq03}
&
u(x,t)=0,& (x,t)\in \pp\Omega \times (0,T),
\\
\label{eq:eq02}
&
u(x,0)=0,& x\in \Omega,
\end{align}
and we investigate the two kinds of inverse problems depending on the way of observations. 

In the inverse source problem via boundary observation, we determine $f(x)$, 
$x\in \Omega$ by spatial data $u(x,t_0)$, $x\in \Omega$ and boundary data on 
$\Sigma$. In the inverse source problem via interior observation, we 
determine $f(x)$, $x\in \Omega$ by spatial data $u(x,t_0)$, $x\in \Omega$ and 
interior data in $Q_\omega$. The main theorems Theorem \ref{thm:ispb} and 
Theorem \ref{thm:ispi} are stated as follows.

\begin{thm}
\label{thm:ispb}
Let us assume that $u,\pp_tu,\pp_t^2u\in\mathcal{U}$ and $u$ satisfies \eqref{eq:eq01}--\eqref{eq:eq02}%and $h_1=0$ in $Q$, $h_2=0$ in $\Omega$
. We suppose that $f\in H^2(\Omega)$ with $f=0$ on $\pp\Omega$ and $\nabla f=0$ on $\gamma$ and $R$ satisfies \eqref{eq:R}.
Then there exist constants $C>0$ such that 
\begin{equation}
\label{eq:seb}
\| f \|_{H^2(\Omega)}
\leq
C
\|u(\cdot, t_0)\|_{H^4(\Omega)}+CB, 
\end{equation}
where 
\begin{equation*}
B=\|\nabla \pp_t^3 u\|_{L^2(\Sigma_\delta)}
+\|\nabla \pp_t^{\frac52} u\|_{L^2(\Sigma_\delta)}
+\|\nabla \pp_t^2 u\|_{L^2(\Sigma_\delta)}
+\|\nabla \pp_t^{\frac32} u\|_{L^2(\Sigma_\delta)}
+\|\nabla \pp_t u\|_{L^2(\Sigma_\delta)}.
\end{equation*}
\end{thm}

\begin{thm}
\label{thm:ispi}
Let us assume that $u,\pp_t u,\pp_t^2u\in\mathcal{U}$ and $u$ satisfies \eqref{eq:eq01}--\eqref{eq:eq02}% and $h_1=0$ in $Q$, $h_2=0$ in $\Omega$
. 
We suppose that $f\in H^2(\Omega)$ with $f=0$ on $\pp\Omega$ and $f=0$ in $\omega$
and $R$ satisfies \eqref{eq:R}.
Then there exist constants $C>0$ such that 
\begin{equation}
\label{eq:sei}
\| f \|_{H^2(\Omega)}
\leq
C
\|u(\cdot, t_0)\|_{H^4(\Omega)}+CI, 
\end{equation}
where 
\begin{equation*}
I=\| \pp_t^3 u\|_{L^2(Q_{\omega,\delta})}
+\|\pp_t^{\frac52} u\|_{L^2(Q_{\omega,\delta})}
+\|\pp_t^2 u\|_{L^2(Q_{\omega,\delta})}
+\|\pp_t^{\frac32} u\|_{L^2(Q_{\omega,\delta})}
+\|\pp_t u\|_{L^2(Q_{\omega,\delta})}.
\end{equation*}
\end{thm}

%\begin{rmk}
%It may be possible to consider the Lipschitz stability by using the idea in \cite{Li-Huang-Yamamoto}, although we need the additional assumptions on the initial data $\pp_t u(x,0)=0$, $x\in\Omega$ and 
%$R(x,0)=0$, $x\in \Omega$. 
%\end{rmk}

\begin{rmk}
There is another approach to obtain the Lipschitz stability in inverse source problems by final observation data. 
In \cite{Sakamoto-Yamamoto11}, 
Sakamoto and Yamamoto considered the perturbation of 
the single term time-fractional diffusion equations with a parameter as the diffusion coefficient
and they obtained the stability estimate 
by means of the analytic perturbation theory 
under the appropriate assumptions on the parameter. 
In our case, however, we may not adopt their methodology directly 
since we consider the diffusion coefficient without the perturbation. 
\end{rmk}

\section{Inverse coefficient problems}
\label{porous}

\subsection{Determination of the zeroth-order coefficient}

Let us consider the inverse problem of determining the zeroth-order coefficient. 
In \eqref{eq:1+1/2eq}, we consider two coefficients $c_k(x)$, $x\in \Omega$ ($k=1,2$), where $c_k\in C^2(\overline{\Omega})$, $c_k(x)\ge0$, $x\in\Omega$ ($k=1,2$). Let $u_k(x,t)$ be the corresponding solutions. We write $L$ as
\begin{equation}
Lu_k(x,t)=Au_k(x,t)-c_k(x)u_k(x,t),
\label{eq:LAck}
\end{equation}
where $A$ is defined as
\[
Au(x,t)=\sum_{i,j=1}^n\pp_i(a_{ij}(x)\pp_ju(x,t))
-\sum_{j=1}^nb_j(x)\pp_ju(x,t), \quad (x,t) \in Q.
\]

By subtraction we obtain
\[
\left\{\begin{aligned}
&
\left(\rho_1\pp_t+\rho_2\pp_t^{\frac12}-A+c_1(x)\right)u(x,t)=f(x)R(x,t),
& (x,t)\in Q,
\\
&u(x,t)=0,
& (x.t)\in \pp\Omega\times(0,T),
\\
&u(x,0)=0,
& x\in \Omega,
\end{aligned}\right.
\]
where
\[
u(x,t)=u_1(x,t)-u_2(x,t),\quad
f(x)=c_1(x)-c_2(x),\quad
R(x,t)=-u_2(x,t)
\]
for $(x,t)\in Q$. 
Thus we arrive at the following two theorems as a direct consequence of the 
inverse source problems. In both cases the Lipschitz stability is obtained. 
Theorem \ref{thm:homo1} is proved using the inverse source problem via 
boundary observation stated in Theorem \ref{thm:ispb}. Theorem 
\ref{thm:homo2} is proved using the inverse source problem via interior 
observation stated in Theorem \ref{thm:ispi}. 

\begin{thm}[boundary observation]
\label{thm:homo1}
Let $u_k,\pp_tu_k,\pp_t^2u_k\in\mathcal{U}$ ($k=1,2$) and $u_1,u_2$ satisfy \eqref{eq:1+1/2eq}--\eqref{eq:ini_condi} with \eqref{eq:LAck}. We suppose that $c_1,c_2\in C^2(\overline{\Omega})
%H^2(\Omega)
$ with $c_1=c_2$ on $\pp\Omega$ and 
$\nabla c_1=\nabla c_2$ on $\gamma$, and $R=-u_2$ satisfies \eqref{eq:R}.
Then there exist constants $C>0$ such that 
\begin{equation}
\label{porous:seb}
\|c_1-c_2\|_{H^2(\Omega)}
\leq
C
\|u_1(\cdot,t_0)-u_2(\cdot,t_0)\|_{H^4(\Omega)}+CB, 
\end{equation}
where 
\begin{align*}
B&=\|\nabla\pp_t^3(u_1-u_2)\|_{L^2(\Sigma_\delta)}
+\|\nabla\pp_t^{\frac52}(u_1-u_2)\|_{L^2(\Sigma_\delta)}
+\|\nabla\pp_t^2(u_1-u_2)\|_{L^2(\Sigma_\delta)}
\\
&
+\|\nabla\pp_t^{\frac32}(u_1-u_2)\|_{L^2(\Sigma_\delta)}
+\|\nabla\pp_t(u_1-u_2)\|_{L^2(\Sigma_\delta)}.
\end{align*}
\end{thm}

\begin{thm}[interior observation]
\label{thm:homo2}
Let $u_k,\pp_tu_k,\pp_t^2u_k\in\mathcal{U}$ ($k=1,2$) and $u_1,u_2$ satisfy \eqref{eq:1+1/2eq}--\eqref{eq:ini_condi} with \eqref{eq:LAck}. We suppose that $c_1,c_2\in C^2(\overline{\Omega})
%H^2(\Omega)
$ 
with $c_1=c_2$ in $\pp\Omega \cup \omega$ and $R=-u_2$ satisfies \eqref{eq:R}.
Then there exist constants $C>0$ such that 
\begin{equation}
\label{porous:sei}
\|c_1-c_2\|_{H^2(\Omega)}
\leq
C\|u_1(\cdot,t_0)-u_2(\cdot,t_0)\|_{H^4(\Omega)}+CI, 
\end{equation}
where 
\begin{align*}
I&=\|\pp_t^3(u_1-u_2)\|_{L^2(Q_{\omega,\delta})}
+\|\pp_t^{\frac52}(u_1-u_2)\|_{L^2(Q_{\omega,\delta})}
+\|\pp_t^2(u_1-u_2)\|_{L^2(Q_{\omega,\delta})}
\\
&+\|\pp_t^{\frac32}(u_1-u_2)\|_{L^2(Q_{\omega,\delta})}
+\|\pp_t(u_1-u_2)\|_{L^2(Q_{\omega,\delta})}.
\end{align*}
\end{thm}

\begin{rmk}
In the case of diffusion in porous media, the condition $|u_2(x,t_0)|=|R(x,t_0)|>0$ for $x\in\overline{\Omega}$ means that the concentration of the target particles is nonzero at the macroscopic scale.
\end{rmk}

\subsection{Determination of the diffusion coefficient}
\label{df}

We consider diffusion coefficients $a_k$ ($k=1,2$) and corresponding solutions $u_k$. Let us express $L$ as
\begin{equation}
Lu(x,t)=\cA_ku(x,t),
\label{eq:LAk}
\end{equation}
where $\cA_k$ is defined as
\[
\cA_k u(x,t)=\dd (a_k (x)\nabla u(x,t))
-\bb(x) \cdot\nabla u(x,t)
-c(x)u(x,t), \quad (x,t) \in Q,
\]
for $k=1,2$. We suppose that $a_k \in C^4(\overline{\Omega})$ $(k=1,2)$,  $\bb=(b_1,b_2,\ldots,b_n) \in \left\{ C^3(\overline{\Omega})\right\}^n$  and $c\in C^3(\overline{\Omega})$. Moreover we assume that there exists a constant $m>0$ such that $a_k(x)\geq m$, $x\in\Omega$ $(k=1,2)$. We investigate the inverse problems of determining the diffusion coefficients $a_k$  ($k=1,2$) by boundary observations and interior observations.

Set 
\[
u(x,t)=u_1(x,t)-u_2(x,t),\quad
a(x)=a_1(x)-a_2(x),\quad
r(x,t)=u_2(x,t)
\]
for $(x,t)\in Q$. Then by subtracting the equations for $k=2$ from ones for $k=1$, we obtain
\begin{equation}
\label{eq:dfeq01}
\left\{\begin{aligned}
&\left(\rho_1\pp_t+\rho_2\pp_t^{\frac12}-\cA_1 \right)u(x,t)=\dd (a(x) \nabla r(x,t)),
&(x,t)\in Q,
\\
&u(x,t)=0, 
&(x,t)\in \pp\Omega\times(0,T),
\\
&u(x,0)=0,
&x\in\Omega,
\end{aligned}\right.
\end{equation}
We assume that 
\begin{equation}
\label{eq:rr}
\left\{
\begin{aligned}
&r\in
C([0,T);C^3 (\overline{\Omega}))\cap
C((0,T);C^5 (\overline{\Omega}))\\
&\qquad
\cap\
C^2((0,T);C^4(\overline{\Omega}))\cap 
C^3((0,T);C^2(\overline{\Omega})), \\
&\pp_t^{\frac12} r \in 
C((0,T);C^3(\overline{\Omega})) \cap
C^2((0,T);C^2(\overline{\Omega})) .
\end{aligned}
\right.
\end{equation}

Let us introduce weight functions for the Carleman estimates introduced in \S\ref{carleman}. According to observation types we consider in this paper, 
we prepare two kinds of distance functions $d_1$ and $d_2$. We choose 
$d_1\in C^2(\overline{\Omega})$ such that
\begin{align*}
&d_1(x)>0, \ x\in\Omega, \quad
|\nabla d_1(x)|>\sigma_1, \ x\in\overline{\Omega},\\
&\sum_{i,j=1}^n a_{ij}(x)\pp_i d_1\nu_j \leq 0,\ x\in \pp\Omega\setminus \gamma,
\end{align*}
where $\sigma_1>0$ is a constant. Let $\omega_0$ be an arbitrarily fixed 
sub-domain of $\Omega$ such that $\omega_0\Subset\omega$. We take 
$d_2 \in C^2(\overline{\Omega})$ such that
\begin{equation*}
d_2(x)>0, \ x\in\Omega, \quad
|\nabla d_2(x)|>\sigma_2, \ x\in\overline{\Omega\setminus \omega_0},\quad
d_2(x)=0, \ x\in \pp\Omega,
\end{equation*}
where $\sigma_2>0$ is a constant. The existence of the distance functions 
$d_1$ and $d_2$ is proved in \cite{FIm, Im, Imanuvilov-Yamamoto98}. Then 
we introduce weight functions $\varphi_k,\psi_k$ ($k=1,2$) as
\begin{equation*}
\varphi_k(x,t)=\frac{e^{\la d_k(x)}}{\ell(t)}, \quad 
\psi_k(x,t)=\frac{e^{\la d_k(x)}-e^{2\la \|d_k\|_{C(\overline{\Omega})}}}{\ell(t)}, \quad (x,t) \in Q,
\end{equation*}
where $\ell(t)=t(T-t)$. Moreover we assume that there exists a constant $m_1>0$ such that 
\begin{equation}
\label{eq:r1}
|\nabla r (x,t_0) \cdot \nabla d_1(x)| \geq m_1, \quad x \in \overline{\Omega},
\end{equation}
or that  there exists a constant $m_2>0$ such that 
\begin{equation}
\label{eq:r2}
|\nabla r (x,t_0) \cdot \nabla d_2(x)| \geq m_2, \quad x \in \overline{\Omega\setminus\omega}.
\end{equation}

Let $D^\prime$ be an arbitrary sub-domain such that $\omega \Subset D^\prime \Subset \Omega$. 
Set $D=\Omega \setminus D^\prime$. 
%Let $D$ be an arbitrarily fixed sub-domain of $\Omega$ such that 
%$\overline{D}\subset\overline{\Omega}$ and $\pp\Omega\subset \pp D$. 
%\revise{Or, let $D$ be the intersection of $\Omega$ and  the union of open neighborhoods of the point on the boundary $\pp\Omega$. }
Henceforth we suppose that $a\equiv 0$ in $D$. 

Now we are ready to state our main results. 

\begin{thm}[boundary observation]
\label{thm:df1}
Let $u_k,\pp_tu_k,\pp_t^2u_k,\nabla u_k\in\mathcal{U}$ ($k=1,2$) and $u_1,u_2$ satisfy \eqref{eq:1+1/2eq}--\eqref{eq:ini_condi} with \eqref{eq:LAk}. We suppose that $a_1,a_2\in C^4(\overline{\Omega})
%H^2(\Omega)
$ with $a_1=a_2$ in $D$ and 
$r=u_2$ satisfies \eqref{eq:rr} and \eqref{eq:r1}.
Then there exist constants $C>0$ such that 
\begin{equation}
\label{df:seb}
\|a_1-a_2\|_{H^3(\Omega)}
\leq
C
\|u_1(\cdot,t_0)-u_2(\cdot,t_0)\|_{H^5(\Omega)}+CB, 
\end{equation}
where 
\begin{align*}
B&=\|\nabla\pp_t^3(u_1-u_2)\|_{L^2(\Sigma_\delta)}
+\|\nabla\pp_t^{\frac52}(u_1-u_2)\|_{L^2(\Sigma_\delta)}
+\|\nabla\pp_t^2(u_1-u_2)\|_{L^2(\Sigma_\delta)}
\\
&
+\|\nabla\pp_t^{\frac32}(u_1-u_2)\|_{L^2(\Sigma_\delta)}
+\|\nabla\pp_t(u_1-u_2)\|_{L^2(\Sigma_\delta)}.
\end{align*}
\end{thm}

\begin{thm}[interior observation]
\label{thm:df2}
Let $u_k,\pp_tu_k,\pp_t^2u_k,\nabla u_k\in\mathcal{U}$ ($k=1,2$) and $u_1,u_2$ satisfy \eqref{eq:1+1/2eq}--\eqref{eq:ini_condi} with \eqref{eq:LAk}. We suppose that $a_1,a_2\in C^4(\overline{\Omega})
%H^2(\Omega)
$ with $a_1=a_2$ in $D\cup \omega$ and 
$r=u_2$ satisfies \eqref{eq:rr} and \eqref{eq:r2}.
Then there exist constants $C>0$ such that 
\begin{equation}
\label{df:sei}
\|a_1-a_2\|_{H^3(\Omega)}
\leq
C\|u_1(\cdot,t_0)-u_2(\cdot,t_0)\|_{H^5(\Omega)}+CI, 
\end{equation}
where 
\begin{align*}
I&=\|\pp_t^3(u_1-u_2)\|_{L^2(Q_{\omega,\delta})}
+\|\pp_t^{\frac52}(u_1-u_2)\|_{L^2(Q_{\omega,\delta})}
+\|\pp_t^2(u_1-u_2)\|_{L^2(Q_{\omega,\delta})}
\\
&+\|\pp_t^{\frac32}(u_1-u_2)\|_{L^2(Q_{\omega,\delta})}
+\|\pp_t(u_1-u_2)\|_{L^2(Q_{\omega,\delta})}.
\end{align*}
\end{thm}
%\begin{rmk}
%As a related work to our inverse problem of determining the diffusion coefficient for fractional differential equation, 
%we may refer to \cite{Ren-Xu14}. 
%They consider single term time- fractional diffusion equation in one dimensional case in space. 
%And they derived the H\"older stability in the inverse coefficient problem. 
%\mm{[Why is \cite{Ren-Xu14} particularly important? What is the relation to \cite{Ren-Xu14}?]}
%\end{rmk}
\begin{rmk}
In one dimensional case in space, we may relax some assumptions on $u_k$ ($k=1,2$). 
It depends on the assumptions of the Carleman estimate for the third order partial differential equations 
(Lemma \ref{lem:ce3rd1} and Lemma \ref{lem:ce3rd2}). 
See also \cite{Ren-Xu14}. 
%\mm{[Maybe we can combine Remark 4 and Remark 5.]}
\end{rmk}
%
%\note{\tiny Need Remark on a posteriori assumptions \eqref{eq:r1} and \eqref{eq:r2}.}
%

\section{Carleman estimate}
\label{carleman}

In this section, we establish the Carleman estimates for \eqref{eq:1+1/2eq}. 
We transform \eqref{eq:1+1/2eq} into an integer-order partial differential 
equation. The calculation is similar to \cite{Xu-Cheng-Yamamoto11}. Let us 
begin with the following Lemma.
\begin{lem}[Lemma 3.1 in \cite{Kwa}]
\label{lem:halftoone}
If $u \in C([0,T];H^4(\Omega))\cap  C^1((0,T);H^2(\Omega))\cap  C^2((0,T);L^2(\Omega))$ satisfies \eqref{eq:1+1/2eq} through \eqref{eq:ini_condi},
then $u$ satisfies
\begin{equation}
\label{eq:lem01}
\rho_2^2 \pp_t u(x,t)- (\rho_1\pp_t - L)^2 u(x,t)=G(x,t),\quad (x,t)\in Q
\end{equation}
where
\begin{equation}
\label{eq:lem02}
G(x,t)=\left[ \rho_2\pp_t^\frac12- (\rho_1\pp_t - L) \right]g(x,t) +\frac{\rho_2 g(x,0)}{\sqrt{\pi t}},\quad (x,t)\in Q.
\end{equation}
\end{lem}
Although $\pp_t^{\frac12}\pp_t^{\frac12} \neq \pp_t$ in general, 
we may obtain the above lemma by applying $\rho_2 \pp_t^{\frac12}-(\rho_1 \pp_t -L)$ 
to the both hand side of \eqref{eq:1+1/2eq} and using $u(x,0)=0$, $x\in \Omega$.

Now we are ready to state our Carleman estimates.

\begin{thm}[Carleman estimate for \eqref{eq:1+1/2eq} with boundary data]
\label{thm:ce0b}
Let $p\geq 0$. 
Suppose that $g(x,t)=0$, $(x,t)\in \pp \Omega\times (0,T)$ and $\nabla g(x,t)=0$, $(x,t)\in \Sigma$. 
Then there exists $\la_0>0$ such that for any $\la>\la_0$, we can choose 
$s_0(\la)>0$ for which there exists $C=C(s_0,\la)>0$ such that
\begin{align}
\label{eq:pceb}
&
\int_Q
\Biggl[
(s\va_1)^{p-1}
\left(
|\pp_t^2 u|^2
+
\sum_{i,j=1}^n|\pp_t\pp_i \pp_j u|^2
\right)
+
(s\va_1)^{p+1}
|\nabla \pp_t u|^2  
\\
&\qquad
+
(s\va_1)^{p+2}
|\nabla (\rho_1 \pp_t -L)u|^2
+
(s\va_1)^{p+3}
\left( |\pp_t u|^2
+\sum_{i,j=1}^n|\pp_i \pp_j u|^2
\right)
\nonumber \\
&\qquad
+
(s\va_1)^{p+5}
 |\nabla u|^2
+
(s\va_1)^{p+7}
|u|^2
\Biggr]
e^{2s\psi_1}\,dxdt
\nonumber\\
&
\leq
C
\int_Q (s\va_1)^{p+1} \left|\left[\rho_2^2\pp_t - (\rho_1\pp_t - L)^2\right]u\right|^2 e^{2s\psi_1}\,dxdt 
\nonumber \\
&\quad
+C 
\int_{\Sigma}
\left[
(s\va_1)^{p+1}
|\nabla \pp_t u|^2 
+
(s\va_1)^{p+2}
|\nabla \pp_t^{\frac12} u|^2
+ 
(s\va_1)^{p+5}
|\nabla u|^2
\right]
e^{2s\psi_1}
\,dSdt ,
\nonumber
\end{align}
for all $s> s_0$ and all $u\in\mathcal{U}$ satisfying \eqref{eq:1+1/2eq} 
with $u(x,t)=0$, $(x,t)\in \pp\Omega \times (0,T)$ and $u(x,0)=0$, $x\in \Omega$. 
%through \eqref{eq:ini_condi}.
\end{thm}

\begin{thm}[Carleman estimate for \eqref{eq:1+1/2eq} with interior data]
\label{thm:ce0i}
Let $p\geq 0$. 
Suppose that $g(x,t)$$=0$, $(x,t)\in \pp \Omega\times (0,T)$ and $ g(x,t)=0$, $(x,t)\in Q_\omega$.
Then there exists $\la_0>0$ such that for any $\la>\la_0$, we can choose 
$s_0(\la)>0$ for which there exists $C=C(s_0,\la)>0$ such that
\begin{align}
\label{eq:pcei}
&
\int_Q
\Biggl[
(s\va_2)^{p-1}
\left(
|\pp_t^2 u|^2
+
\sum_{i,j=1}^n|\pp_t\pp_i \pp_j u|^2
\right)
+
(s\va_2)^{p+1}
|\nabla \pp_t u|^2  
\\
&\qquad
+
(s\va_2)^{p+2}
|\nabla (\rho_1 \pp_t -L)u|^2
+
(s\va_2)^{p+3}
\left( |\pp_t u|^2
+\sum_{i,j=1}^n|\pp_i \pp_j u|^2
\right)
\nonumber \\
&\qquad
+
(s\va_2)^{p+5}
 |\nabla u|^2
+
(s\va_2)^{p+7}
|u|^2
\Biggr]
e^{2s\psi_2}\,dxdt
\nonumber\\
&
\leq
C
\int_Q (s\va_2)^{p+1} \left|\left[\rho_2^2\pp_t - (\rho_1\pp_t - L)^2\right]u\right|^2 e^{2s\psi_2}\,dxdt 
\nonumber \\
&\quad
+C 
\int_{Q_\omega}
\left[
(s\va_2)^{p+3}
| \pp_t u|^2
+ 
(s\va_2)^{p+4}
|\pp_t^{\frac12} u|^2 
+
(s\va_2)^{p+7}
|u|^2
\right]
e^{2s\psi_2}
\,dxdt,
\nonumber
\end{align}
for all $s> s_0$ and all $u\in\mathcal{U}$ satisfying \eqref{eq:1+1/2eq} 
with $u(x,t)=0$, $(x,t)\in \pp\Omega \times (0,T)$ and $u(x,0)=0$, $x\in \Omega$. 
%through \eqref{eq:ini_condi}.
\end{thm}

To prove Theorems \ref{thm:ce0b} and \ref{thm:ce0i}, we start with 
the global Carleman estimates for parabolic equations (see e.g., 
\cite{Im, Yamamoto09}) stated in Lemmas \ref{lem:celempb} and 
\ref{lem:celempi} below.

\begin{lem}
\label{lem:celempb}
Let $p\geq 0$. 
There exists $\la_0>0$ such that for any $\la>\la_0$, we can choose 
$s_0(\la)>0$ for which there exists $C=C(s_0,\la)>0$ such that
\begin{align*}
&
\int_Q
\left[
(s\va_1)^{p-1}
\left(
|\pp_t v|^2
+
\sum_{i,j=1}^n|\pp_i \pp_j v|^2
\right)
+
(s\va_1)^{p+1}
|\nabla v|^2
+
(s\va_1)^{p+3}
|v|^2
\right]
\!
e^{2s\psi_1}\,dxdt \\
&
\leq
C\int_Q
(s\va_1)^{p}
|(\rho_1\pp_t -L) v|^2 e^{2s\psi_1}\,dxdt 
+ 
C
\int_{\Sigma}
(s\va_1)^{p+1}|\nabla v|^2 
e^{2s\psi_1}
\,dSdt,
\end{align*}
for all $s> s_0$ and all $v \in L^2(0,T;H^2(\Omega))\cap  H^1(0,T;L^2(\Omega))$ satisfying $v (x,t)=0$, $(x,t)\in \pp\Omega\times(0,T)$.
\end{lem}

\begin{lem}
\label{lem:celempi}
Let $p\geq 0$. There exists $\la_0>0$ such that for any $\la>\la_0$, we can 
choose $s_0(\la)>0$ for which there exists $C=C(s_0,\la)>0$ such that
\begin{align*}
&
\int_Q
\left[
(s\va_2)^{p-1}
\left(
|\pp_t v|^2
+
\sum_{i,j=1}^n|\pp_i \pp_j v|^2
\right)
+
(s\va_2)^{p+1}
|\nabla v|^2
+
(s\va_2)^{p+3}
|v|^2
\right]
\!
e^{2s\psi_2}\,dxdt \\
&
\leq
C\int_Q
(s\va_2)^{p}
|(\rho_1\pp_t -L) v|^2 e^{2s\psi_2}\,dxdt 
+ 
C
\int_{Q_\omega}
(s\va_2)^{p+3}|v|^2 
e^{2s\psi_2}
\,dxdt,
\end{align*}
for all $s> s_0$ and all $v \in L^2(0,T;H^2(\Omega))\cap  H^1(0,T;L^2(\Omega))$ satisfying $v (x,t)=0$, $(x,t)\in \pp\Omega\times(0,T)$.
\end{lem}

\begin{proof}[Proof of Theorem \ref{thm:ce0b}]
Throughout the proof, we assume that $s>1$ is large enough to satisfy $s\va>1$ in $Q$. 

Equation \eqref{eq:lem01} yields
\begin{equation}
\label{eq:ce01}
\rho_1\pp_t w(x,t)-Lw(x,t)=\rho_2^2\pp_tu(x,t)-G(x,t),\quad(x,t)\in Q,
\end{equation}
where
\begin{equation}
w(x,t)=\rho_1\pp_t u(x,t)-Lu(x,t),\quad(x,t)\in Q.
\label{eq:ce02}
\end{equation}
Since $u(x,t)=0$, $(x,t)\in \pp\Omega\times (0,T)$ and 
$g(x,t)=0$, $(x,t) \in \pp\Omega\times (0,T)$, we have by \eqref{eq:1+1/2eq},
\begin{equation*}
w(x,t)
=\rho_1\pp_t u(x,t) - L u(x,t)  
=g(x,t)-\rho_2\pp_t^{\frac12} u(x,t)
=0, \quad
(x,t)\in\pp\Omega\times (0,T).
\end{equation*}
Applying the Lemma \ref{lem:celempb} to \eqref{eq:ce01}, we obtain
\begin{align}
\label{eq:ce03}
&\int_Q
\left[
(s\va_1)^{p_1-1} |\pp_t w|^2
+
(s\va_1)^{p_1+1}
|\nabla w|^2
+
(s\va_1)^{p_1+3}
|w|^2
\right]e^{2s\psi_1}\,dxdt
\\
&\leq
C\int_Q (s\va_1)^{p_1}|\pp_t u|^2 e^{2s\psi_1}\,dxdt
+C\int_Q (s\va_1)^{p_1}|G|^2 e^{2s\psi_1}\,dxdt
\nonumber \\
&\quad
+C
\int_{\Sigma}
(s\va_1)^{p_1+1}|\nabla w|^2 e^{2s\psi_1}
\,dSdt,
\nonumber
\end{align}
for $p_1\ge0$. Next by applying Lemma \ref{lem:celempb} to \eqref{eq:ce02}, we obtain
\begin{align}
\label{eq:ce04}
&\int_Q
\!
\Biggl[
(s\va_1)^{p_2-1}
\left(
|\pp_t u|^2
+\sum_{i,j=1}^n |\pp_i \pp_j u|^2
\right) \\
&\qquad
+(s\va_1)^{p_2+1}
|\nabla u|^2
+(s\va_1)^{p_2+3}
|u|^2
\Biggr]
\!
e^{2s\psi_1}\,dxdt
\nonumber \\
&\leq
C\int_Q (s\va_1)^{p_2}|w|^2 e^{2s\psi_1}\,dxdt
+C
\int_{\Sigma}
(s\va_1)^{p_2+1} |\nabla u|^2 
e^{2s\psi_1}
\,dSdt,
\nonumber
\end{align}
for $p_2\geq 0$.

Putting $p_2=p_1+1$ and substituting the estimate of 
$|\pp_t u|^2$ in \eqref{eq:ce04} into the right-hand side of \eqref{eq:ce03}, 
we obtain
\begin{align}
&\int_Q
\left[
(s\va_1)^{p_1-1} 
|\pp_t w|^2
+
(s\va_1)^{p_1+1}
|\nabla w|^2
+
(s\va_1)^{p_1+3}
|w|^2
\right]e^{2s\psi_1}\,dxdt
\nonumber \\
&\leq
C\int_Q (s\va_1)^{p_1+1}|w|^2 e^{2s\va_1}\,dxdt
+C\int_Q (s\va_1)^{p_1} 
|G|^2 e^{2s\psi_1}\,dxdt
+C B_{1,p_1},
\nonumber
\end{align}
where
\begin{equation*}
B_{1,p_1}=
\int_{\Sigma}
\left[
(s\va_1)^{p_1+1} 
|\nabla w|^2
+
(s\va_1)^{p_1+2} 
|\nabla u|^2
\right]
e^{2s\psi_1}
\,dSdt
.
\end{equation*}
Taking sufficiently large $s>0$, we can absorb the first term on the right-hand side of the above inequality into the left-hand side and we have
\begin{align}
\label{eq:ce05}
&\int_Q
\left[
(s\va_1)^{p_1-1} 
|\pp_t w|^2
+
(s\va_1)^{p_1+1}
|\nabla w|^2
+
(s\va_1)^{p_1+3}
|w|^2
\right]e^{2s\psi_1}\,dxdt \\
&\leq
C\int_Q (s\va_1)^{p_1} 
|G|^2 e^{2s\psi_1}\,dxdt
+C B_{1,p_1}.
\nonumber
\end{align}
By \eqref{eq:ce04} with $p_2=p_1+3$ and \eqref{eq:ce05}, we obitain
\begin{align}
\label{eq:ce06}
&\int_Q
\Biggl[
(s\va_1)^{p_1+1}
|\nabla (\rho_1 \pp_t -L)u|^2
+
(s\va_1)^{p_1+2}
\left(
|\pp_t u|^2
+\sum_{i,j=1}^n |\pp_i \pp_j u|^2
\right) \\
&\qquad
+(s\va_1)^{p_1+4}
|\nabla u|^2
+(s\va_1)^{p_1+6}
|u|^2
\Biggr]
e^{2s\psi_1}\,dxdt \nonumber \\
&
\leq
C\int_Q (s\va_1)^{p_1} |G|^2 e^{2s\psi_1}\,dxdt
+CB_{2,p_1},
\nonumber
%\nonumber \\
%&\mm{+
%C\int_Q(s\va_1)^{p+3}|w|^2e^{2s\psi_1}\,dxdt,?}
%\nonumber
\end{align}
where
\begin{equation*}
B_{2,p_1}=
\int_{\Sigma}
\left[
(s\va_1)^{p_1+1} 
|\nabla w|^2
+
(s\va_1)^{p_1+4} 
|\nabla u|^2
\right]
e^{2s\psi_1}
\,dSdt.
\end{equation*}

Let us choose $p_1=p+1$ in \eqref{eq:ce05}. Then from \eqref{eq:ce02} and \eqref{eq:ce05}, we have
\begin{equation*}
%\label{eq:ce07}
\int_Q
(s\va_1)^{p}
|\pp_t (\rho_1\pp_t u -L u)|^2
e^{2s\psi_1}\,dxdt 
\leq
C\int_Q (s\va_1)^{p+1} |G|^2 e^{2s\psi_1}\,dxdt
+CB_{1,p+1}.
\end{equation*}
%By differentiating both sides of \eqref{eq:ce02} with respect to $t$, applying \eqref{eq:ce05}, and 
Setting $u_0=\pp_t u$, we obtain
\begin{equation}
\label{eq:ce10}
\int_Q
(s\va_1)^{p}|\rho_1\pp_t u_0 -L u_0|^2e^{2s\psi_1}\,dxdt 
\leq
C\int_Q (s\va_1)^{p+1} |G|^2 e^{2s\psi_1}\,dxdt
+CB_{1,p+1}.
\end{equation}
If we use Lemma \ref{lem:celempb} with $v=u_0$ and applying \eqref{eq:ce10}, 
we obtain
\begin{align*}
&
\int_Q
\Biggl[
(s\va_1)^{p-1}
\left(
|\pp_t u_0|^2
+
\sum_{i,j=1}^n |\pp_i\pp_j u_0|^2
\right) \\
&\qquad
+
(s\va_1)^{p+1}
|\nabla u_0|^2
+
(s\va_1)^{p+3}
|u_0|^2
\Biggr]e^{2s\psi_1}\,dxdt\\
&\leq
C\int_Q (s\va_1)^{p+1} |G|^2 e^{2s\psi_1}\,dxdt
+C
\int_{\Sigma}
(s\va_1)^{p+1}
 |\nabla u_0|^2 
 e^{2s\psi_1}
\,dSdt
+CB_{1,p+1}.
\end{align*}
Recalling $u_0=\pp_tu$, we have
\begin{align}
\nonumber
&\int_Q
\Biggl[
(s\va_1)^{p-1}
\left(
|\pp_t^2 u|^2
+
\sum_{i,j=1}^n |\pp_t\pp_i\pp_j u|^2
\right) \\
&\qquad
+
(s\va_1)^{p+1}
|\nabla \pp_t u|^2
+
(s\va_1)^{p+3}
|\pp_t u|^2
\Biggr]e^{2s\psi_1}\,dxdt
\nonumber \\
&\leq
C\int_Q 
(s\va_1)^{p+1} |G|^2 e^{2s\psi_1}\,dxdt
+C B_{3,p},
\nonumber
\end{align}
where 
\begin{equation*}
B_{3,p}=
\int_{\Sigma}
\left[
(s\va_1)^{p+1}
|\nabla\pp_t u|^2
+
(s\va_1)^{p+2} 
|\nabla w|^2
+
(s\va_1)^{p+3} 
|\nabla u|^2
\right]
e^{2s\psi_1}
\,dSdt.
\end{equation*}
Hence using \eqref{eq:ce06},  we obtain
\begin{align*}
&
\int_Q
\Biggl[
(s\va_1)^{p-1}
\left(
|\pp_t^2 u|^2
+
\sum_{i,j=1}^n|\pp_t\pp_i \pp_j u|^2
\right)
+
(s\va_1)^{p+1}
|\nabla \pp_t u|^2  
\\
&\qquad
+
(s\va_1)^{p+2}
|\nabla (\rho_1 \pp_t -L)u|^2
+
(s\va_1)^{p+3}
\left( |\pp_t u|^2
+\sum_{i,j=1}^n|\pp_i \pp_j u|^2
\right)
\\
&\qquad+
(s\va_1)^{p+5}
 |\nabla u|^2
+
(s\va_1)^{p+7}
|u|^2
\Biggr]
e^{2s\psi_1}\,dxdt \\
&
\leq
C
\int_Q (s\va_1)^{p+1} \left|G\right|^2 e^{2s\psi_1}\,dxdt
+C B_{4,p},
%\\
%&\mm{+
%C\int_Q(s\va_1)^4|w|^2e^{2s\psi_1}\,dxdt?
%},
\end{align*}
where
\begin{equation*}
B_{4,p}=
\int_{\Sigma}
\left[
(s\va_1)^{p+1}
|\nabla\pp_t u|^2
+
(s\va_1)^{p+2} 
|\nabla w|^2
+
(s\va_1)^{p+5} 
|\nabla u|^2
\right]
e^{2s\psi_1}
\,dSdt.
\end{equation*}
Finally, we consider the boundary term $B_4$. Since $\nabla g=0$ on $\Sigma$ is assumed,  $\nabla w=\nabla g-\rho_2\nabla \pp_t^{\frac12} u=-\rho_2\nabla\pp_t^{\frac12}u$ on $\Sigma$. Hence we have
\begin{align}
\nonumber
&
\int_Q
\Biggl[
(s\va_1)^{p-1}
\left(
|\pp_t^2 u|^2
+
\sum_{i,j=1}^n|\pp_t\pp_i \pp_j u|^2
\right)
+
(s\va_1)^{p+1}
|\nabla \pp_t u|^2  \\
&\qquad
+
(s\va_1)^{p+2}
|\nabla (\rho_1 \pp_t -L)u|^2
+
(s\va_1)^{p+3}
\left( |\pp_t u|^2
+\sum_{i,j=1}^n|\pp_i \pp_j u|^2
\right)
\nonumber \\
&\qquad
+
(s\va_1)^{p+5}
 |\nabla u|^2
+
(s\va_1)^{p+7}
|u|^2
\Biggr]
e^{2s\psi_1}\,dxdt
\nonumber\\
&
\leq
C
\int_Q (s\va_1)^{p+1} \left|G\right|^2 e^{2s\psi_1}\,dxdt 
\nonumber \\
&\quad
+C 
\int_{\Sigma}
\left[
(s\va_1)^{p+1}
|\nabla \pp_t u|^2 
+
(s\va_1)^{p+2}
|\nabla \pp_t^{\frac12} u|^2
+ 
(s\va_1)^{p+5}
|\nabla u|^2
\right]
e^{2s\psi_1}
\,dSdt .
\nonumber
\end{align}
Thus we obtain \eqref{eq:pceb}.
\end{proof}

\begin{proof}[Proof of Theorem \ref{thm:ce0i}] 
By using Lemma \ref{lem:celempi} instead of Lemma \ref{lem:celempb}, we can 
prove Theorem \ref{thm:ce0i} in the same way as Theorem \ref{thm:ce0b}. 
\end{proof}

Furthermore we need Carleman estimates for elliptic equations in the proof of 
the stability estimates in inverse source problems which we will develop in \S\ref{proof}.

Let us assume that $\wa_{ij}\in C^1(\overline{\Omega})$, $\wa_{ij}= \wa_{ji}$ ($i,j=1,\ldots,n$), 
$\wb_j \in C(\overline{\Omega})$ ($j=1,\ldots,n$), 
$\wc \in C(\overline{\Omega})$, and 
that there exists a constant $\wm>0$ such that
\begin{equation*}
\frac1{\wm} |\xi|^2
\leq
\sum_{i,j=1}^n \wa_{ij}(x) \xi_i \xi_j
\leq
\wm |\xi|^2,
\quad
\xi=(\xi_1,\ldots, \xi_n) \in \R^n,\
x \in \overline{\Omega}.
\end{equation*}
We consider the following symmetric uniformly elliptic operator.
\begin{equation*}
\widetilde{L} \wv(x) :=\sum_{i,j=1}^n \pp_i (\wa_{ij}(x) \pp_j \wv(x))
-\sum_{j=1}^n \wb_j (x)\pp_j \wv(x)
- \wc(x)\wv(x),\ x\in \Omega.
\end{equation*}
Set $\wva_k(x):=\va_k(x,t_0)$, $x\in \Omega$ and 
$\wpsi_k(x):=\psi_k(x,t_0)$, $x\in \Omega$ for $k=1,2$.
Then we have the following Lemmas.
%\mm{[Cite related papers.]}

\begin{lem}
\label{lem:celemeb}
Let $p \geq 0$. There exists $\la_0>0$ such that for any $\la>\la_0$, we can 
choose $s_0(\la)>0$ for which there exists $C=C(s_0,\la)>0$ such that
\begin{align*}
&
\int_\Omega
\left[
(s\wva_1)^{p-1}
\sum_{i,j=1}^n|\pp_i \pp_j \wv|^2
+
(s\wva_1)^{p+1}
|\nabla \wv|^2
+
(s\wva_1)^{p+3}
 |\wv|^2
\right]
e^{2s \wpsi_1}\,dx
\\
&
\leq
C\int_\Omega (s\wva_1)^{p} |\widetilde{L} \wv|^2 e^{2s \wpsi_1}\,dx
+
C
\int_{\gamma}
(s\wva_1)^{p+1}|\nabla \wv|^2 e^{2s\wpsi_1}
\,dS,
\end{align*}
for all $s> s_0$ and all $\wv \in H^2(\Omega)$ satisfying $\wv (x)=0$, $x\in \pp\Omega$. 
\end{lem}

\begin{lem}
\label{lem:celemei}
Let $p \geq 0$. There exists $\la_0>0$ such that for any $\la>\la_0$, we can 
choose $s_0(\la)>0$ for which there exists $C=C(s_0,\la)>0$ such that
\begin{align*}
&
\int_\Omega
\left[
(s\wva_2)^{p-1}
\sum_{i,j=1}^n|\pp_i \pp_j \wv|^2
+
(s\wva_2)^{p+1}
|\nabla \wv|^2
+
(s\wva_2)^{p+3}
 |\wv|^2
\right]
e^{2s \wpsi_2}\,dx
\\
&
\leq
C\int_\Omega (s\wva_2)^{p} |\widetilde{L} \wv|^2 e^{2s \wpsi_2}\,dx
+
C
\int_{\omega}
(s\wva_2)^{p+3}|\wv|^2 e^{2s\wpsi_2}
\,dx,
\end{align*}
for all $s> s_0$ and all $\wv \in H^2(\Omega)$ satisfying $\wv (x)=0$, $x\in \pp\Omega$.
\end{lem}
These lemmas can be shown in the same manner as the parabolic case by means of integration by parts. 
Hence we omit the proofs of these lemmas here.

We conclude  this section by introducing Carleman estimates for the third order partial differential equations which we use in the proof of the stability estimates in inverse problems of determining the diffusion coefficients. 

Let $\bp=(p_1, \ldots, p_2)\in \{C^1(\overline{\Omega})\}^n$. 

\begin{lem}
\label{lem:ce3rd1}
We assume that there exists $m_1 >0$ such that 
$|\bp (x)\cdot \nabla d_1(x)|\geq m_1$, $x\in \overline{\Omega}$. 
Then there exists $\la_0>0$ such that for any $\la>\la_0$, 
we can choose $s_0(\la)>0$ for which there exists $C=C(s_0,\la)>0$ such that 
\begin{align*}
&
\int_\Omega
\Biggl[
s\wva_1
\sum_{i,j,k=1}^n|\pp_i \pp_j \pp_k \wv|^2
+
(s\wva_1)^{2}
|\nabla \triangle \wv|^2
\\
&\qquad
+
(s\wva_1)^{3}
\sum_{i,j=1}^n|\pp_i \pp_j \wv|^2
+(s\wva_1)^{5}
\left(|\nabla \wv|^2+ |\wv|^2 \right)
\Biggr]
e^{2s \wpsi_1}\,dx
\\
&
\leq
C\int_\Omega \left(  |\nabla (\bp \cdot \nabla \triangle \wv)|^2 + |\bp \cdot \nabla \triangle \wv|^2\right) e^{2s \wpsi_1}\,dx,
\end{align*}
for all $s> s_0$ and all $\wv \in H^4(\Omega)$ satisfying 
$|\wv(x)|=|\nabla \wv(x)|=|\triangle \wv(x)|=|\nabla \triangle \wv(x)|=0$, $x\in \pp\Omega$  
and 
$|\nabla \pp_k \wv(x)|=0$, $x\in \gamma$ ($k=1,2,\ldots, n$). 
%$|\nabla \triangle \wv(x)|=|\triangle \wv(x)|=0$, $x\in \{x\in \pp\Omega\mid \bp \cdot \nu \geq 0\}$. 
\end{lem}

\begin{lem}
\label{lem:ce3rd2}
We assume that there exists $m_2 >0$ such that 
$|\bp (x)\cdot \nabla d_2(x)|\geq m_2$, $x\in \overline{\Omega\setminus \omega}$. 
Then there exists $\la_0>0$ such that for any $\la>\la_0$, 
we can choose $s_0(\la)>0$ for which there exists $C=C(s_0,\la)>0$ such that 
\begin{align*}
&
\int_\Omega
\Biggl[
s\wva_2
\sum_{i,j,k=1}^n|\pp_i \pp_j \pp_k \wv|^2
+
(s\wva_2)^{2}
|\nabla \triangle \wv|^2
\\
&\qquad
+
(s\wva_2)^{3}
\sum_{i,j=1}^n|\pp_i \pp_j \wv|^2
+(s\wva_2)^{5}
\left(|\nabla \wv|^2+ |\wv|^2 \right)
\Biggr]
e^{2s \wpsi_2}\,dx
\\
&
\leq
C\int_\Omega \left(  |\nabla (\bp \cdot \nabla \triangle \wv)|^2 + |\bp \cdot \nabla \triangle \wv|^2\right) e^{2s \wpsi_2}\,dx,
\end{align*}
for all $s> s_0$ and all $\wv \in H^4(\Omega)$ satisfying 
$|\wv(x)|=|\nabla \wv(x)|=|\triangle \wv(x)|=|\nabla \triangle \wv(x)|=0$, $x\in \pp\Omega$ 
and 
$\wv(x)=0$, $x\in \omega$.
%$|\nabla \triangle \wv(x)|=|\triangle \wv(x)|=0$, $x\in \{x\in \pp\Omega\mid \bp \cdot \nu \geq 0\}$. 
\end{lem}

To establish the  Carleman estimate for  $\bp \cdot \nabla \triangle \wv$, 
we start by proving the first order partial differential equations $\bp \cdot \nabla \wv$. 

\begin{lem}
\label{lem:ce1st1}
We assume that there exists $m_1 >0$ such that 
$|\bp (x)\cdot \nabla d_1(x)|\geq m_1$, $x\in \overline{\Omega}$. 
Then there exists $\la_0>0$ such that for any $\la>\la_0$, 
we can choose $s_0(\la)>0$ for which there exists $C=C(s_0,\la)>0$ such that 
\begin{equation*}
\int_\Omega
(s\wva_1)^{2}
\left(
|\nabla \wv|^2
+
|\wv|^2
\right)
e^{2s \wpsi_1}\,dx
\leq
C\int_\Omega \left(  |\nabla (\bp \cdot \nabla \wv)|^2 + |\bp \cdot \nabla  \wv|^2\right) e^{2s \wpsi_1}\,dx,
\end{equation*}
for all $s> s_0$ and all $\wv \in H^2(\Omega)$ satisfying 
$|\wv(x)|=|\nabla \wv(x)|=0$, $x\in \pp\Omega$.
%\{x\in \pp\Omega\mid \bp \cdot \nu \geq 0\}$. 
\end{lem}

\begin{lem}
\label{lem:ce1st2}
We assume that there exists $m_2 >0$ such that 
$|\bp (x)\cdot \nabla d_2(x)|\geq m_2$, $x\in \overline{\Omega\setminus \omega}$. 
Then there exists $\la_0>0$ such that for any $\la>\la_0$, 
we can choose $s_0(\la)>0$ for which there exists $C=C(s_0,\la)>0$ such that 
\begin{equation*}
\int_\Omega
(s\wva_2)^{2}
\left(
|\nabla \wv|^2
+
|\wv|^2
\right)
e^{2s \wpsi_2}\,dx
\leq
C\int_\Omega \left(  |\nabla (\bp \cdot \nabla \wv)|^2 + |\bp \cdot \nabla  \wv|^2\right) e^{2s \wpsi_2}\,dx,
\end{equation*}
for all $s> s_0$ and all $\wv \in H^2(\Omega)$ satisfying 
$|\wv(x)|=|\nabla \wv(x)|=0$, $x\in \pp\Omega$ 
%\{x\in \pp\Omega\mid \bp \cdot \nu \geq 0\}$ 
and $\wv(x)=0$, $x\in\omega$. 
\end{lem}

\begin{proof}[Proof of Lemma \ref{lem:ce1st1}]
Setting $\ww=\wv e^{s\wpsi_1}$ in $\Omega$, we have
\begin{equation*}
e^{s\wpsi_1}(\bp\cdot \nabla \wv)
=\bp\cdot \nabla \ww -s\la\wva_1 (\bp\cdot \nabla d_1)\ww
\quad \text{in $\Omega$}. 
\end{equation*}
Taking the weighted $L^2$ norm, by integrating by parts we obtain 
\begin{align*}
&
\int_\Omega 
|\bp\cdot \nabla \wv|^2 e^{2s\wpsi_1}
\,dx \\
&=
\int_\Omega 
|\bp\cdot \nabla \ww|^2 
\,dx
+
\int_\Omega 
s^2\la^2 \wva_1^2
(\bp\cdot \nabla d_1)^2
|\ww|^2 
\,dx \\
&\quad
-2
\int_\Omega 
s\la\wva_1 (\bp\cdot \nabla d_1) \sum_{j=1}^n p_j \ww \pp_j \ww 
\,dx \\
&\geq
\int_\Omega 
s^2\la^2 \wva_1^2
(\bp\cdot \nabla d_1)^2
|\ww|^2 
\,dx 
-
\int_\Omega 
s\la\wva_1 (\bp\cdot \nabla d_1) \sum_{j=1}^n p_j \pp_j (\ww)^2 
\,dx \\
&=
\int_\Omega 
s^2\la^2 \wva_1^2
(\bp\cdot \nabla d_1)^2
|\ww|^2 
\,dx 
+
\int_\Omega 
s\la^2 \wva_1
(\bp\cdot \nabla d_1)^2
|\ww|^2 
\,dx
\\
&\quad
+
\int_\Omega 
s\la\wva_1 
\left[
(\bp\cdot \nabla d_1)
(\dd \bp)
+
\bp\cdot \nabla (\bp\cdot \nabla d_1)
\right]
 |\ww|^2 
\,dx
\end{align*}
Hence we have
\begin{equation*}
\int_\Omega 
s^2\la^2 \wva_1^2
|\ww|^2 
\,dx 
\leq
C\int_\Omega 
|\bp\cdot \nabla \wv|^2 e^{2s\wpsi_1}
\,dx 
+C
\int_\Omega 
\left(s\la^2 \wva_1 + s\la \wva_1 \right)
|\ww|^2 
\,dx .
\end{equation*}
Taking sufficiently large $s>0$, 
we may absorb the second term on the right-hand side of the above inequality into the left-hand side and we see that
\begin{equation*}
\int_\Omega 
s^2\la^2 \wva_1^2
|\ww|^2 
\,dx 
\leq
C\int_\Omega 
|\bp\cdot \nabla \wv|^2 e^{2s\wpsi_1}
\,dx ,
\end{equation*}
that is, 
\begin{equation}
\label{eq:ce1st01}
\int_\Omega 
s^2\la^2 \wva_1^2
|\wv|^2 e^{2s\wpsi_1}
\,dx 
\leq
C\int_\Omega 
|\bp\cdot \nabla \wv|^2 e^{2s\wpsi_1}
\,dx .
\end{equation}

Set $\wv_k =\pp_k \wv$ in $\Omega$ for $k=1,2\ldots, n$. 
We consider 
\begin{equation*}
\bp\cdot \nabla \wv_k
=
\pp_k (\bp\cdot \nabla \wv)
-(\pp_k \bp) \cdot \nabla \wv
\end{equation*}
Applying the estimate \eqref{eq:ce1st01} to the above equation, 
we may obtain
\begin{align*}
\int_\Omega 
s^2\la^2 \wva_1^2
|\wv_k|^2 e^{2s\wpsi_1}
\,dx 
&\leq
C\int_\Omega 
|\bp\cdot \nabla \wv_k|^2 e^{2s\wpsi_1}
\,dx \\
&\leq
C\int_\Omega 
|\pp_k (\bp\cdot \nabla \wv)|^2 e^{2s\wpsi_1}
\,dx 
+
C\int_\Omega 
|(\pp_k \bp) \cdot \nabla \wv|^2 e^{2s\wpsi_1}
\,dx .
\end{align*}
Hence we have
\begin{align*}
\int_\Omega 
s^2\la^2 \wva_1^2
|\nabla \wv|^2 e^{2s\wpsi_1}
\,dx 
&\leq
C\int_\Omega 
|\nabla (\bp\cdot \nabla \wv)|^2 e^{2s\wpsi_1}
\,dx 
+
C\int_\Omega 
|\nabla \wv|^2 e^{2s\wpsi_1}
\,dx .
\end{align*}
Choosing sufficiently large $s>0$, 
we can absorb the second term on the right-hand side of the above inequality into the left-hand side and we may get 
\begin{equation*}
\int_\Omega 
s^2\la^2 \wva_1^2
|\nabla \wv|^2 e^{2s\wpsi_1}
\,dx 
\leq
C\int_\Omega 
|\nabla (\bp\cdot \nabla \wv)|^2 e^{2s\wpsi_1}
\,dx.
\end{equation*}
Combining this with \eqref{eq:ce1st01}, 
we obtain the Carleman estimate of Lemma \ref{lem:ce1st1}. 
Thus we conclude the Lemma \ref{lem:ce1st1}. 
\end{proof}

\begin{proof}[Proof of Lemma \ref{lem:ce1st2}] 
By an argument similar to that used in the proof of Lemma \ref{lem:ce1st1}, we may obtain Lemma \ref{lem:ce1st2}. 
\end{proof}

\begin{rmk}
In one spatial dimension, the assumption $|\nabla \wv|=0$ on $\pp\Omega$ is not necessary in Lemma \ref{lem:ce1st1} and Lemma \ref{lem:ce1st2}. In this case, we have the following Carleman estimate by integration by parts.
\begin{equation*}
\int_\Omega
(s\wva_k)^{2}
\left(
|\pp_1 \wv|^2
+
|\wv|^2
\right)
e^{2s \wpsi_k}\,dx
\leq
C\int_\Omega  |p_1\pp_1  \wv|^2 e^{2s \wpsi_k}\,dx,
\end{equation*}
for $k=1,2$. 
\end{rmk}

Now we are ready to prove Lemma \ref{lem:ce3rd1} and Lemma \ref{lem:ce3rd2}. 

\begin{proof}[Proof of Lemma \ref{lem:ce3rd1}]
Set $\wy=\triangle \wv$ in $\Omega$. 
By the assumptions $|\triangle \wv(x)|=|\nabla \triangle \wv(x)|=0$, $x\in \pp\Omega$, we see that 
$|\wy(x)|=|\nabla \wy(x)|=0$, $x\in \pp\Omega$. 
By Lemma \ref{lem:ce1st1}, we obtain 
\begin{equation*}
\int_\Omega
(s\wva_1)^{2}
\left(
|\nabla \wy|^2
+
|\wy|^2
\right)
e^{2s \wpsi_1}\,dx
\leq
C\int_\Omega \left(  |\nabla (\bp \cdot \nabla \wy)|^2 + |\bp \cdot \nabla  \wy|^2\right) e^{2s \wpsi_1}\,dx,
\end{equation*}
that is,
\begin{align}
\label{eq:ce3rd01}
&\int_\Omega
(s\wva_1)^{2}
\left(
|\nabla \triangle \wv|^2
+
|\triangle \wv|^2
\right)
e^{2s \wpsi_1}\,dx \\
&\leq
C\int_\Omega \left(  |\nabla (\bp \cdot \nabla \triangle \wv)|^2 + |\bp \cdot \nabla  \triangle \wv|^2\right) e^{2s \wpsi_1}\,dx. \nonumber 
\end{align}
Next we use the Carleman estimate for elliptic equations to estimate 
the left-hand side of the above inequality. 
By Lemma \ref{lem:celemeb} with $p=2$, 
we have
\begin{align}
\label{eq:ce3rd02}
&
\int_\Omega
\left[
s\wva_1
\sum_{i,j=1}^n|\pp_i \pp_j \wv|^2
+
(s\wva_1)^{3}
|\nabla \wv|^2
+
(s\wva_1)^{5}
 |\wv|^2
\right]
e^{2s \wpsi_1}\,dx
\\
&
\leq
C\int_\Omega (s\wva_1)^{2} |\triangle \wv|^2 e^{2s \wpsi_1}\,dx
. \nonumber 
\end{align}
Setting $\wv_k =\pp_k \wv$ in $\Omega$ for $k=1,2\ldots, n$ 
and using Lemma \ref{lem:celemeb} again, 
we see that 
\begin{align*}
&
\int_\Omega
\left[
s\wva_1
\sum_{i,j=1}^n|\pp_i \pp_j \wv_k|^2
+
(s\wva_1)^{3}
|\nabla \wv_k|^2
+
(s\wva_1)^{5}
 |\wv_k|^2
\right]
e^{2s \wpsi_1}\,dx
\\
&
\leq
C\int_\Omega (s\wva_1)^{2} |\triangle \wv_k|^2 e^{2s \wpsi_1}\,dx,
\end{align*}
that is, 
\begin{align}
\label{eq:ce3rd03}
&
\int_\Omega
\left[
s\wva_1
\sum_{i,j,k=1}^n|\pp_i \pp_j \pp_k \wv|^2
+
(s\wva_1)^{3}
\sum_{i,j=1}^n
|\pp_i \pp_j  \wv|^2
+
(s\wva_1)^{5}
 |\nabla \wv|^2
\right]
e^{2s \wpsi_1}\,dx
\\
&
\leq
C\int_\Omega (s\wva_1)^{2} |\nabla \triangle \wv|^2 e^{2s \wpsi_1}\,dx. 
\nonumber 
\end{align}
Summing up the inequalities \eqref{eq:ce3rd01} --\eqref{eq:ce3rd03}, 
we may obtain the Carleman estimate of Lemma \ref{lem:ce3rd1}
\end{proof}

\begin{proof}[Proof of Lemma \ref{lem:ce3rd2}] 
Using Lemma \ref{lem:ce1st2} and Lemma \ref{lem:celemei} 
instead of Lemma \ref{lem:ce1st1} and Lemma \ref{lem:celemeb}, 
we may prove Lemma \ref{lem:ce3rd2} in the same way as Lemma \ref{lem:ce3rd1}. 
\end{proof}

\section{Proof of stability estimates}
\label{proof}

Hereafter we let $C$ denote a generic constant which is independent of $s,x,t$ and let
$C(s)$ denote a generic constant which is independent of $x,t$ but depends on $s$.

\subsection{Stability for the zeroth-order coefficient}

\begin{proof}[Proof of Theorem \ref{thm:ispb}]

By using Lemma \ref{lem:halftoone} for \eqref{eq:eq01}, we obtain
\begin{equation}
\label{eq:pr01}
\rho_2^2\pp_t u (x,t)
-(\rho_1\pp_t -L)^2 u(x,t)
=F(x,t),\quad (x,t)\in Q,
\end{equation}
where we introduced $F(x,t)$ as
\begin{align}
\label{eq:pr02}
F(x,t)
&=\left[ \rho_2\pp_t^\frac12- (\rho_1\pp_t - L) \right] \left(f(x)R(x,t)\right) +\rho_2 f(x)\frac{R(x,0)}{\sqrt{\pi t}}\\
&=
R(x,t)
\sum_{i,j=1}^n \pp_i(a_{ij}(x)\pp_j f(x)) \nonumber\\
&\quad
+\sum_{j=1}^n \left( 2 \sum_{i=1}^n a_{ij}(x)\pp_i R(x,t) - b_j(x) R(x,t)\right) \pp_j f(x) \nonumber\\
&\quad
+\Biggl[
\rho_2\pp_t^{\frac12}R(x,t)-\rho_1\pp_t R(x,t)
+\sum_{i,j=1}^n \pp_i(a_{ij}(x)\pp_j R(x,t)) \nonumber \\
&\qquad\quad
-\sum_{j=1}^n b_j(x) \pp_j R(x,t)-c(x) R(x,t)
+
\frac{\rho_2 R(x,0)}{\sqrt{\pi t}}
\Biggr]f(x),\  (x,t) \in Q
\nonumber .
\end{align}

Let us set $y=\pp_t u$, $z=\pp_t^2 u$ in $Q$. By differentiating 
\eqref{eq:pr01} with respect to $t$, we have
\begin{align}
\label{eq:pr06}
&
\rho_2^2\pp_t y (x,t)
-(\rho_1\pp_t -L)^2 y(x,t)
=\pp_t F(x,t),		& (x,t)\in Q,\\
\label{eq:pr07}
&
\rho_2^2\pp_t z (x,t)
-(\rho_1\pp_t -L)^2 z(x,t)
=\pp_t^2 F(x,t),	& (x,t)\in Q.
\end{align}
Since $u(x,t)=0$, $(x,t)\in\pp\Omega\times(0,T)$, we see that
\begin{equation*}
y(x,t)=z(x,t)=0,\quad (x,t)\in\pp\Omega\times(0,T).
\end{equation*}

To use the Carleman estimate in $Q_\delta$, we introduce the weight functions. 
Set
\begin{equation*}
\varphi_{\delta,1}(x,t)=\frac{e^{\la d_1(x)}}{\ell_\delta(t)}, \quad 
\psi_{\delta,1}(x,t)=\frac{e^{\la d_1(x)}-e^{2\la \|d_1\|_{C(\overline{\Omega})}}}{\ell_\delta(t)}, \quad (x,t) \in Q_\delta,
\end{equation*}
where $\ell_\delta(t)=(t-t_0+\delta)(t_0+\delta-t)$. 
Fixing $\la>0$ and applying Theorem \ref{thm:ce0b} ($p=0$) to \eqref{eq:pr06} and \eqref{eq:pr07} in $Q_\delta$, we have
\begin{align}
\label{eq:pr08}
&
\int_{Q_\delta}
\Biggl[
(s\va_{\delta,1})^3
\left( 
|\pp_t y|^2
+
|\pp_t z|^2
+\sum_{i,j=1}^n|\pp_i \pp_j y|^2
+\sum_{i,j=1}^n|\pp_i \pp_j z|^2
\right)  \\
&\qquad
+
(s\va_{\delta,1})^5
\left(
 |\nabla y|^2
 + |\nabla z|^2
\right)
+
(s\va_{\delta,1})^7
\left(
|y|^2
+
|z|^2
\right)
\Biggr]
e^{2s\psi_{\delta,1}}\,dxdt
\nonumber\\
&
\leq
C
\int_{Q_\delta} s\va_{\delta,1} \left(
|\pp_t F|^2
+
|\pp_t^2 F|^2
\right) e^{2s\psi_{\delta,1}}\,dxdt
+C \widetilde{B},
\nonumber
\end{align}
where
\begin{align*}
\widetilde{B}=&  \int_{\Sigma_\delta}
\Biggl[
s\va_{\delta,1}
\left(
|\nabla \pp_t y|^2 
+
|\nabla \pp_t z|^2 
\right)\\
&\qquad
+
(s\va_{\delta,1})^2
\left(
|\nabla \pp_t^{\frac12} y|^2
+
|\nabla \pp_t^{\frac12} z|^2
\right) 
+ 
(s\va_{\delta,1})^5
\left(
|\nabla y|^2
+
|\nabla z|^2
\right)
\Biggr]
e^{2s\psi_{\delta,1}}
\,dSdt.
\end{align*}
We can estimate $\widetilde{B}$ by $B^2$. 
We note that 
$\pp_t^{\frac12}\pp_t^m=\pp_t^{m+\frac12}$, $m\in\mathbb{N}$. 
Since there exist constants $C_k(s)>0$ such that  $\va_{\delta,1}^ke^{2s\psi_{\delta,1}} \leq C_k(s)$ on $\Sigma_\delta$ for $k=0,1,2,\ldots$, we have
\begin{align*}
\widetilde{B}&\leq
C
\int_{\Sigma_\delta}
\Biggl[
s\va_{\delta,1}
|\nabla \pp_t^3 u|^2 
+
(s\va_{\delta,1})^2
\left(
|\nabla \pp_t^{\frac12} \pp_t u|^2
+
|\nabla \pp_t^{\frac12} \pp_t^2 u|^2
\right) \\
&\qquad\qquad
+ 
(s\va_{\delta,1})^5
\left(
|\nabla \pp_t u|^2
+
|\nabla \pp_t^2 u|^2
\right)
\Biggr]
e^{2s\psi_{\delta,1}}
\,dSdt \\
&\leq C(s) B^2.
\end{align*}
Note that
\begin{equation*}
\int_{Q_\delta} s\va_{\delta,1} \left(
|\pp_t F|^2
+
|\pp_t^2 F|^2
\right) e^{2s\psi_{\delta,1}}\,dxdt
\leq 
C
\int_{Q_\delta} 
s\va_{\delta,1} \sum_{|\alpha| \leq 2} |\pp_x^\alpha f|^2e^{2s\psi_{\delta,1}}\,dxdt. 
\end{equation*}
This together with \eqref{eq:pr08} gives
\begin{align}
\label{eq:pr09}
&
\int_{Q_\delta}
\Biggl[
(s\va_{\delta,1})^3
\left( 
|\pp_t y|^2
+
|\pp_t z|^2
+\sum_{i,j=1}^n|\pp_i \pp_j y|^2
+\sum_{i,j=1}^n|\pp_i \pp_j z|^2
\right) \\
&\qquad
+
(s\va_{\delta,1})^5
\left(
 |\nabla y|^2
 + |\nabla z|^2
\right)
+
(s\va_{\delta,1})^7
\left(
|y|^2
+
|z|^2
\right)
\Biggr]
e^{2s\psi_{\delta,1}}\,dxdt
\nonumber\\
&
\leq
C
\int_{Q_\delta} 
s\va_{\delta,1} \sum_{|\alpha |\leq 2} |\pp_x^\alpha f|^2e^{2s\psi_{\delta,1}}\,dxdt
+C(s)B^2.
\nonumber
\end{align}

Let us expand the left-hand side of \eqref{eq:pr01}. We have
\begin{equation*}
%\label{eq:pr03}
\rho_2^2\pp_t u (x,t)
-\rho_1^2\pp_t^2 u(x,t) +2\rho_1\pp_t L u(x,t) -L^2 u(x,t)
=F(x,t),
\quad
(x,t)\in Q.
\end{equation*}
In particular at $t=t_0$, we have
\begin{equation}
\label{eq:pr04}
\rho_2^2\pp_t u (x,t_0)
-\rho_1^2\pp_t^2 u(x,t_0) +2\rho_1\pp_t L u(x,t_0) -L^2 u(x,t_0)
=F(x,t_0),\quad x\in \Omega.
\end{equation}
Taking the weighted $L^2$ norm of \eqref{eq:pr04} in $\Omega$, we obtain
\begin{align}
\label{eq:pr05}
&
\int_{\Omega} \va_{\delta,1}(x,t_0) |F(x,t_0)|^2e^{2s\psi_{\delta,1}(x,t_0)}\,dx
\\
&\leq
C \sum_{k=1}^3 J_k+
C
\int_{\Omega} \va_{\delta,1}(x,t_0) \sum_{|\alpha|\leq 4} |\pp_x^\alpha u(x,t_0)|^2e^{2s\psi_{\delta,1}(x,t_0)}\,dx,
\nonumber
\end{align}
where
\begin{align}
J_1&=
\int_{\Omega} \va_{\delta,1}(x,t_0) |\pp_t u(x,t_0)|^2e^{2s\psi_{\delta,1}(x,t_0)}\,dx,
\nonumber \\
J_2&=
\int_{\Omega} \va_{\delta,1}(x,t_0) |\pp_t^2 u(x,t_0)|^2e^{2s\psi_{\delta,1}(x,t_0)}\,dx,
\nonumber \\
J_3&=
\int_{\Omega} \va_{\delta,1}(x,t_0) |\pp_t L u(x,t_0)|^2e^{2s\psi_{\delta,1}(x,t_0)}\,dx.
\nonumber
\end{align}
Let us estimate $J_1,J_2,J_3$. We assume that $s>1$ is large enough to satisfy $s\va_{\delta,1}>1$ in $Q$. 
We note that $\pp_t \psi_{\delta,1}(x,t)=(e^{2\la (\| d_1\|_{C(\overline{\Omega})} - d_1(x))} -e^{-\la d_1(x)})(T-2t)\va_{\delta,1}^2(x,t)$ for $(x,t) \in Q$. 
\begin{align*}
J_1
&=
\int_{t_0-\delta}^{t_0}
\int_{\Omega}
\pp_t \left( \va_{\delta,1}|y|^2 e^{2s\psi_{\delta,1}}\right)
\,dxdt  \\
&\leq
C
\int_{t_0-\delta}^{t_0}
\int_{\Omega}
\left[ \va_{\delta,1}^2|y|^2+ 
\va_{\delta,1}|\pp_t y||y| + 
s\va_{\delta,1}^3|y|^2 \right]e^{2s\psi_{\delta,1}}
\,dxdt \\
&\leq
C
\int_{Q_\delta}
s\va_{\delta,1}^3
\left( 
|y|^2+|z|^2
\right)
e^{2s\psi_{\delta,1}}
\,dxdt .
\end{align*}
Combining this with \eqref{eq:pr09}, we may estimate the right-hand side of the above inequality and we obtain
\begin{equation}
\label{eq:pr13}
J_1
\leq
\frac{C}{s^5}
\int_{Q_\delta} 
\va_{\delta,1} \sum_{|\alpha |\leq 2} |\pp_x^\alpha f|^2e^{2s\psi_{\delta,1}}\,dxdt
+C(s)B^2.
\end{equation}
Similarly, we obtain
\begin{align*}
J_2
&=
\int_{t_0-\delta}^{t_0}
\int_{\Omega}
\pp_t \left( \va_{\delta,1}|\pp_t y|^2 e^{2s\psi_{\delta,1}}\right)
\,dxdt  \\
&\leq
C
\int_{t_0-\delta}^{t_0}
\int_{\Omega}
\left[ \va_{\delta,1}^2|\pp_ty|^2+ 
\va_{\delta,1}|\pp_t^2 y||\pp_ty| + 
s\va_{\delta,1}^3|\pp_t y|^2 \right]e^{2s\psi_{\delta,1}}
\,dxdt \\
&\leq
C
\int_{Q_\delta}
s\va_{\delta,1}^3
\left( 
|\pp_t y|^2+|\pp_t z|^2
\right)
e^{2s\psi_{\delta,1}}
\,dxdt.
\end{align*}
Putting this together with \eqref{eq:pr09}, we see that
\begin{equation}
\label{eq:pr14}
J_2
\leq
\frac{C}{s}
\int_{Q_\delta} 
\va_{\delta,1} \sum_{|\alpha |\leq 2} |\pp_x^\alpha f|^2e^{2s\psi_{\delta,1}}\,dxdt
+C(s) B^2.
\end{equation}
Moreover we have
\begin{align*}
J_3
&=
\int_{t_0-\delta}^{t_0}
\int_{\Omega}
\pp_t \left( \va_{\delta,1}|L y|^2 e^{2s\psi_{\delta,1}}\right)
\,dxdt  \\
&\leq
C
\int_{t_0-\delta}^{t_0}
\int_{\Omega}
\left[ \va_{\delta,1}^2|L y|^2+ 
\va_{\delta,1}|\pp_t L y||L y| + 
s\va_{\delta,1}^3|L y|^2 \right]e^{2s\psi_{\delta,1}}
\,dxdt \\
&\leq
C
\int_{Q_\delta}
s\va_{\delta,1}^3
\left( 
|L y|^2+|L z|^2
\right)
e^{2s\psi_{\delta,1}}
\,dxdt  \\
&\leq
C
\int_{Q_\delta}
s\va_{\delta,1}^3
\left( 
\sum_{|\alpha| \leq 2}
|\pp_x^\alpha y|^2
+
\sum_{|\alpha|\leq 2}
|\pp_x^\alpha z|^2
\right)
e^{2s\psi_{\delta,1}}
\,dxdt  
\end{align*}
This together with \eqref{eq:pr09} gives
\begin{equation}
\label{eq:pr15}
J_3
\leq
\frac{C}{s}
\int_{Q_\delta} 
\va_{\delta,1} \sum_{|\alpha |\leq 2} |\pp_x^\alpha f|^2e^{2s\psi_{\delta,1}}\,dxdt
+C(s) B^2.
\end{equation}

By \eqref{eq:pr05} through \eqref{eq:pr15}, we have
\begin{align}
\label{eq:pr16}
&\int_{\Omega} \va_{\delta,1}(x,t_0) |F(x,t_0)|^2e^{2s\psi_{\delta,1}(x,t_0)}\,dx \\
&
\leq
\frac{C}{s}
\int_{Q_\delta} 
\va_{\delta,1} \sum_{|\alpha |\leq 2} |\pp_x^\alpha f|^2e^{2s\psi_{\delta,1}}\,dxdt
\nonumber \\
&\quad
+
C
\int_{\Omega} \va_{\delta,1}(x,t_0) \sum_{|\alpha|\leq 4} |\pp_x^\alpha u(x,t_0)|^2e^{2s\psi_{\delta,1}(x,t_0)}\,dx +C(s) B^2.
\nonumber
\end{align}
We will estimate the left-hand side of the inequality \eqref{eq:pr16} from 
below using the Carleman estimate for the elliptic equation stated 
Lemma \ref{lem:celemeb} ($p=1$). By \eqref{eq:pr02} at $t=t_0$, we have
\begin{align}
\label{eq:pr17}
&
\sum_{i,j=1}^n \pp_i(a_{ij}(x)\pp_j \wf(x)) \\
&\quad
+\frac{1}{R(x,t_0)}\sum_{j=1}^n \left( 2 \sum_{i=1}^n a_{ij}(x)\pp_i R(x,t_0) - b_j(x) R(x,t_0) \right) \pp_j \wf(x) \nonumber\\
&\quad
+\frac{1}{R(x,t_0)}\Biggl[
\rho_2\pp_t^{\frac12}R(x,t_0)-\rho_1\pp_t R(x,t_0)
+\sum_{i,j=1}^n \pp_i(a_{ij}(x)\pp_j R(x,t_0)) \nonumber \\
&\qquad\qquad\qquad
-\sum_{j=1}^n b_j(x) \pp_j R(x,t_0)-c(x) R(x,t_0)
+
\frac{\rho_2 R(x,0)}{\sqrt{\pi t_0}}
\Biggr]\wf(x) \nonumber \\
&=
\frac{F(x,t_0)}{R(x,t_0)}, \quad x\in \Omega.
\nonumber
\end{align}
We note that $f(x)=0$, $x\in \pp\Omega$ and $\nabla f(x)=0$, $x\in \gamma$ 
are assumed. Applying the Lemma \ref{lem:celemeb} to \eqref{eq:pr17} in 
$\Omega$, we obtain
\begin{align}
\label{eq:pr18}
&
\frac1{s}
\int_{\Omega}
\va_{\delta,1}(x,t_0) \sum_{|\alpha|\leq 2}  | \pp_x^\alpha f(x) |^2 e^{2s\psi_{\delta,1}(x,t_0)}
\,dx \\
&\leq
\frac{C}{s}
\int_{\Omega}
 \sum_{|\alpha|\leq 2}  | \pp_x^\alpha f(x) |^2 e^{2s\psi_{\delta,1}(x,t_0)}
\,dx \nonumber \\
&\leq
\frac{C}{s}
\int_{\Omega}
\Biggl(
\sum_{i,j=1}^n 
|\pp_i \pp_j f(x)|^2
\nonumber \\
&\qquad \qquad 
+ 
(s\va_{\delta,1}(x,t_0))^2
|\nabla f(x)|^2
+
(s\va_{\delta,1}(x,t_0))^4
|f(x)|^2
\Biggr)
e^{2s\psi_{\delta,1}(x,t_0)}\,dx  \nonumber \\
&\leq 
C
\int_{\Omega}
\va_{\delta,1}(x,t_0)
\left|
\frac{F(x,t_0)}{R(x,t_0)}
\right|^2
e^{2s\psi (x,t_0)}\,dx
\nonumber \\
&\leq 
C
\int_{\Omega}
\va_{\delta,1}(x,t_0)
\left|
F(x,t_0)
\right|^2
e^{2s\psi (x,t_0)}\,dx. 
\nonumber 
\end{align}

By \eqref{eq:pr16} and \eqref{eq:pr18}, we obtain
\begin{align}
\label{eq:pr19}
&
\frac1{s}
\int_{\Omega}
\va_{\delta,1}(x,t_0) \sum_{|\alpha|\leq 2}  | \pp_x^\alpha f(x) |^2 e^{2s\psi_{\delta,1}(x,t_0)}
\,dx \\
&
\leq
\frac{C}{s}
\int_{Q_\delta} 
\va_{\delta,1}(x,t_0) \sum_{|\alpha |\leq 2} |\pp_x^\alpha f(x)|^2e^{2s\psi_{\delta,1}(x,t_0)}\,dxdt
\nonumber \\
&\quad
+
C
\int_{\Omega} \va_{\delta,1}(x,t_0) \sum_{|\alpha|\leq 4} |\pp_x^\alpha u(x,t_0)|^2e^{2s\psi_{\delta,1}(x,t_0)}\,dx +C(s) B^2.
\nonumber
\end{align}
Let us estimate the first integral term on the right-hand side of \eqref{eq:pr19}.
\begin{equation*}
\int_{Q_\delta} 
\va_{\delta,1} \sum_{|\alpha |\leq 2} |\pp_x^\alpha f|^2e^{2s\psi_{\delta,1}}\,dxdt
\leq
\int_{\Omega} \va_{\delta,1}(x,t_0) \sum_{|\alpha|\leq 2} | \pp_x^\alpha f |^2 e^{2s\psi_{\delta,1}(x,t_0)}  h_s(x)\,dx,
\end{equation*}
where
\begin{equation*}
h_s(x)=
\frac{1}{\va_{\delta,1}(x,t_0)}\int_{t_0-\delta}^{t_0+\delta} \va_{\delta,1}e^{-2s(\psi_{\delta,1}(x,t_0)-\psi_{\delta,1}(x,t))}\,dt.
\end{equation*}
Since $\psi_{\delta,1}(x,t_0)-\psi_{\delta,1}(x,t) \geq 0$, 
$(x,t)\in Q_\delta$, $h_s$ converges pointwise to $0$ in $\Omega$ as 
$s\to\infty$ by Lebesgue's dominated convergence theorem. Moreover by 
Dini's theorem, we see that $h_s$ converges uniformly to $0$ in $\Omega$ as 
$s\to \infty$. Hence, taking sufficiently large $s>0$, we can absorb the 
first term on the right-hand side of \eqref{eq:pr19} 
into the left-hand side and obtain
\begin{align}
\label{eq:pr20}
&\frac1{s}
\int_{\Omega}
\va_{\delta,1}(x,t_0) \sum_{|\alpha|\leq 2}  | \pp_x^\alpha f(x) |^2 e^{2s\psi_{\delta,1}(x,t_0)}
\,dx \\
&
\leq
C
\int_{\Omega} \va_{\delta,1}(x,t_0) \sum_{|\alpha|\leq 4} |\pp_x^\alpha u(x,t_0)|^2e^{2s\psi_{\delta,1}(x,t_0)}\,dx +C(s) B^2.
\nonumber
\end{align}
Fix $s>0$. Noting that $ \va_{\delta,1}(\cdot,t_0)e^{2s\psi_{\delta,1}(\cdot,t_0)}$ has its upper and lower bound in $\overline{\Omega}$, we see that 
\begin{equation*}
\|f \|_{H^2(\Omega)} 
\leq C \| u (\cdot, t_0) \|_{H^4(\Omega)} +CB.
\end{equation*}
Thus we obtain the stability estimate.  
\end{proof}
\begin{proof}[Proof of Theorem \ref{thm:ispi}] 
We may prove Theorem \ref{thm:ispi} by an argument similar to that used in the proof of Theorem \ref{thm:ispb}. In the proof, 
Theorem \ref{thm:ce0i} and Lemma \ref{lem:celemei} are used 
instead of Theorem \ref{thm:ce0b} and Lemma \ref{lem:celemeb}.
\end{proof}

\subsection{Stability for the diffusion coefficient}

Next we prove Theorem \ref{thm:df1} and Theorem \ref{thm:df2}. 
The proofs is very similar to the proofs of Theorem \ref{thm:ispb} and Theorem \ref{thm:ispi}. 

\begin{proof}[Proof of Theorem \ref{thm:df1}]
Applying Lemma \ref{lem:halftoone} to \eqref{eq:dfeq01}, we obtain
\begin{equation}
\label{eq:dfpr01}
\rho_2^2\pp_t u (x,t)
-(\rho_1\pp_t -\cA_1)^2 u(x,t)
=\wF(x,t),\quad (x,t)\in Q,
\end{equation}
where
\begin{align}
\label{eq:dfpr02}
\wF(x,t)
&=\left[ \rho_2\pp_t^\frac12- (\rho_1\pp_t - \cA_1) \right] \left(\dd  (a(x)\nabla r(x,t))\right) 
+
\frac{\rho_2 \dd (a(x)\nabla r(x,0))}{\sqrt{\pi t}}
\\
&=
a_1(x) \nabla r(x,t)\cdot \nabla \triangle a(x)
+2a_1(x) \sum_{i,j=1}^n (\pp_i \pp_j r(x,t)) (\pp_i\pp_j a(x))
\nonumber \\
&\quad
+a_1(x) \triangle r(x,t) \triangle a(x)
+(\nabla a_1(x)-\bb(x))\cdot (\nabla r(x,t)\cdot \nabla)\nabla a(x)
\nonumber \\
&\quad
+\Biggl[
(\rho_2 \pp_t^{\frac12} -\rho_1 \pp_t)\nabla r(x,t) 
+3 a_1(x)\nabla \triangle r(x,t) 
+\left(\triangle r(x,t)\right) \nabla a_1(x)
\nonumber \\
&\qquad
-\left(\triangle r(x,t)\right) \bb(x)
-c(x) \nabla r(x,t)
+\frac{\rho_2 \nabla r(x,0)}{\sqrt{\pi t}}
\Biggr]\cdot \nabla a(x) 
\nonumber \\
&\quad 
+(\nabla a_1(x)-\bb(x))\cdot (\nabla a(x)\cdot \nabla) \nabla r(x,t)
\nonumber \\
&\quad
+\Biggl[
(\rho_2 \pp_t^{\frac12} -\rho_1 \pp_t)\triangle r(x,t) 
+\left(\nabla  a_1(x)\cdot \nabla \triangle r(x,t)\right) 
+a_1(x)\triangle^2 r(x,t) 
\nonumber \\
&\qquad
-\left(\bb(x) \cdot \nabla\triangle r(x,t) \right)
-c(x) \triangle r(x,t)
+
\frac{\rho_2 \triangle r(x,0)}{\sqrt{\pi t}}
\Biggr]a(x),\quad
 (x,t) \in Q
\nonumber .
\end{align}

Setting $y=\pp_t u$, $z=\pp_t^2 u$ in $Q$ and differentiating 
\eqref{eq:dfpr01} with respect to $t$, we have
\begin{align}
\label{eq:dfpr06}
&
\rho_2^2\pp_t y (x,t)
-(\rho_1\pp_t -\cA_1)^2 y(x,t)
=\pp_t \wF(x,t),		& (x,t)\in Q,\\
\label{eq:dfpr07}
&
\rho_2^2\pp_t z (x,t)
-(\rho_1\pp_t -\cA_1)^2 z(x,t)
=\pp_t^2 \wF(x,t),	& (x,t)\in Q.
\end{align}
Since $u(x,t)=0$, $(x,t)\in\pp\Omega\times(0,T)$, we see that
\begin{equation*}
y(x,t)=z(x,t)=0,\quad (x,t)\in\pp\Omega\times(0,T).
\end{equation*}

Fixing $\la>0$ and applying Theorem \ref{thm:ce0b} ($p=1$) to \eqref{eq:dfpr06} and \eqref{eq:dfpr07} in $Q_\delta$, we have
\begin{align}
\label{eq:dfpr08}
&
\int_{Q_\delta}
\Biggl[
(s\va_{\delta,1})^2
\left( 
|\nabla\pp_t y|^2
+
|\nabla\pp_t z|^2
\right)  
\\
&\qquad
+
(s\va_{\delta,1})^3
\left(
|\nabla(\rho_1\pp_t - \cA_1)y|^2
+
|\nabla(\rho_1\pp_t - \cA_1)z|^2
\right)
\nonumber \\
&\qquad
+
(s\va_{\delta,1})^4
\left( 
|\pp_t y|^2
+
|\pp_t z|^2
+\sum_{i,j=1}^n|\pp_i \pp_j y|^2
+\sum_{i,j=1}^n|\pp_i \pp_j z|^2
\right)  \nonumber\\
&\qquad
+
(s\va_{\delta,1})^6
\left(
 |\nabla y|^2
 + |\nabla z|^2
\right)
+
(s\va_{\delta,1})^8
\left(
|y|^2
+
|z|^2
\right)
\Biggr]
e^{2s\psi_{\delta,1}}\,dxdt
\nonumber\\
&
\leq
C
\int_{Q_\delta} 
(s\va_{\delta,1})^2
\left(
|\pp_t \wF|^2
+
|\pp_t^2 \wF|^2
\right) e^{2s\psi_{\delta,1}}\,dxdt
+C \widehat{B},
\nonumber
\end{align}
where
\begin{align*}
\widehat{B}=&  \int_{\Sigma_\delta}
\Biggl[
(s\va_{\delta,1})^2
\left(
|\nabla \pp_t y|^2 
+
|\nabla \pp_t z|^2 
\right)\\
&\qquad
+
(s\va_{\delta,1})^3
\left(
|\nabla \pp_t^{\frac12} y|^2
+
|\nabla \pp_t^{\frac12} z|^2
\right) 
+ 
(s\va_{\delta,1})^6
\left(
|\nabla y|^2
+
|\nabla z|^2
\right)
\Biggr]
e^{2s\psi_{\delta,1}}
\,dSdt.
\end{align*}
As we have seen in the proof of Theorem \ref{thm:ispb}, we may obtain $\widehat{B}\leq C(s) B^2$.

Note that 
\begin{align*}
&\int_{Q_\delta} 
(s\va_{\delta,1})^2
\left(
|\pp_t \wF|^2
+
|\pp_t^2 \wF|^2
\right) e^{2s\psi_{\delta,1}}\,dxdt \\
&
\leq 
C
\int_{Q_\delta} 
(s\va_{\delta,1})^2
\left(
|\nabla \triangle a|^2
+ 
\sum_{|\alpha| \leq 2} |\pp_x^\alpha a|^2
\right)
e^{2s\psi_{\delta,1}}\,dxdt. 
\end{align*}
This together with \eqref{eq:dfpr08} gives
\begin{align}
\label{eq:dfpr09}
&
\int_{Q_\delta}
\Biggl[
(s\va_{\delta,1})^2
\left( 
|\nabla\pp_t y|^2
+
|\nabla\pp_t z|^2
\right)  
\\
&\qquad
+
(s\va_{\delta,1})^3
\left(
|\nabla(\rho_1\pp_t - \cA_1)y|^2
+
|\nabla(\rho_1\pp_t - \cA_1)z|^2
\right)
\nonumber \\
&\qquad
+
(s\va_{\delta,1})^4
\left( 
|\pp_t y|^2
+
|\pp_t z|^2
+\sum_{i,j=1}^n|\pp_i \pp_j y|^2
+\sum_{i,j=1}^n|\pp_i \pp_j z|^2
\right)  \nonumber\\
&\qquad
+
(s\va_{\delta,1})^6
\left(
 |\nabla y|^2
 + |\nabla z|^2
\right)
+
(s\va_{\delta,1})^8
\left(
|y|^2
+
|z|^2
\right)
\Biggr]
e^{2s\psi_{\delta,1}}\,dxdt
\nonumber\\
&
\leq
C
\int_{Q_\delta} 
(s\va_{\delta,1})^2
\left(
|\nabla \triangle a|^2
+ 
\sum_{|\alpha| \leq 2} |\pp_x^\alpha a|^2
\right)
e^{2s\psi_{\delta,1}}\,dxdt
+C(s) B^2.
\nonumber
\end{align}

Let us expand the left-hand side of \eqref{eq:dfpr01}. We have
\begin{equation*}
\rho_2^2\pp_t u (x,t)
-\rho_1^2\pp_t^2 u(x,t) +2\rho_1\pp_t \cA_1 u(x,t) -\cA_1^2 u(x,t)
=\wF(x,t),
\quad
(x,t)\in Q.
\end{equation*}
Moreover we have
\begin{equation*}
\rho_2^2\nabla\pp_t u (x,t)
-\rho_1^2\nabla\pp_t^2 u(x,t) +2\rho_1\nabla\pp_t \cA_1 u(x,t) -\nabla\cA_1^2 u(x,t)
=\nabla\wF(x,t),
\ 
(x,t)\in Q.
\end{equation*}

In particular at $t=t_0$, we have
\begin{equation}
\label{eq:dfpr03}
\rho_2^2\pp_t u (x,t_0)
-\rho_1^2\pp_t^2 u(x,t_0) +2\rho_1\pp_t \cA_1 u(x,t_0) -\cA_1^2 u(x,t_0)
=\wF(x,t_0),\  x\in \Omega, 
\end{equation}
and
\begin{equation}
\label{eq:dfpr04}
\rho_2^2\nabla\pp_t u (x,t_0)
-\rho_1^2\nabla\pp_t^2 u(x,t_0) +2\rho_1\nabla\pp_t \cA_1 u(x,t_0) -\nabla\cA_1^2 u(x,t_0)
=\nabla\wF(x,t_0),\  x\in \Omega.
\end{equation}
Taking the weighted $L^2$ norm of \eqref{eq:dfpr03} and \eqref{eq:dfpr04} in $\Omega$, we obtain
\begin{align}
\label{eq:dfpr05}
&
\int_{\Omega} \left(|\wF(x,t_0)|^2+|\nabla\wF(x,t_0)|^2\right)e^{2s\psi_{\delta,1}(x,t_0)}\,dx
\\
&\leq
C \sum_{k=1}^6 \wJ_k+
C
\int_{\Omega}  \sum_{|\alpha|\leq 5} |\pp_x^\alpha u(x,t_0)|^2e^{2s\psi_{\delta,1}(x,t_0)}\,dx,
\nonumber
\end{align}
where
\begin{align}
&\wJ_1=
\int_{\Omega} |\pp_t u(x,t_0)|^2e^{2s\psi_{\delta,1}(x,t_0)}\,dx,
&&\wJ_2=
\int_{\Omega} |\pp_t^2 u(x,t_0)|^2e^{2s\psi_{\delta,1}(x,t_0)}\,dx,
\nonumber \\
&\wJ_3=
\int_{\Omega} |\pp_t \cA_1 u(x,t_0)|^2e^{2s\psi_{\delta,1}(x,t_0)}\,dx,
&&\wJ_4=
\int_{\Omega} |\nabla \pp_t u(x,t_0)|^2e^{2s\psi_{\delta,1}(x,t_0)}\,dx,
\nonumber \\
&\wJ_5=
\int_{\Omega} |\nabla \pp_t^2 u(x,t_0)|^2e^{2s\psi_{\delta,1}(x,t_0)}\,dx,
&&\wJ_6=
\int_{\Omega} |\nabla \pp_t \cA_1 u(x,t_0)|^2e^{2s\psi_{\delta,1}(x,t_0)}\,dx.
\nonumber
\end{align}
Henceforth we estimate $\wJ_1$ through $\wJ_6$ by using the Carleman estimate. 
We assume that $s>1$ is large enough to satisfy $s\va_{\delta,1}>1$ in $Q$. 
We note that $\pp_t \psi_{\delta,1}(x,t)=(e^{2\la (\| d_1\|_{C(\overline{\Omega})} - d_1(x))} -e^{-\la d_1(x)})(T-2t)\va_{\delta,1}^2(x,t)$ for $(x,t) \in Q$. 
\begin{align*}
\wJ_1
&=
\int_{t_0-\delta}^{t_0}
\int_{\Omega}
\pp_t \left( |y|^2 e^{2s\psi_{\delta,1}}\right)
\,dxdt  \\
&\leq
C
\int_{t_0-\delta}^{t_0}
\int_{\Omega}
\left(
|\pp_t y||y| + 
s\va_{\delta,1}^2
|y|^2 \right)e^{2s\psi_{\delta,1}}
\,dxdt \\
&\leq
C
\int_{Q_\delta}
s\va_{\delta,1}^2
\left( 
|y|^2+|z|^2
\right)
e^{2s\psi_{\delta,1}}
\,dxdt .
\end{align*}
Combining this with \eqref{eq:dfpr09}, we may estimate the right-hand side of the above inequality and we obtain
\begin{equation}
\label{eq:dfpr101}
\wJ_1
\leq
\frac{C}{s^5}
\int_{Q_\delta} 
\va_{\delta,1}^2
\left(
|\nabla \triangle a|^2
+ 
\sum_{|\alpha| \leq 2} |\pp_x^\alpha a|^2
\right)
e^{2s\psi_{\delta,1}}\,dxdt
+C(s) B^2.
\end{equation}
Similarly, we obtain
\begin{align*}
\wJ_2
&=
\int_{t_0-\delta}^{t_0}
\int_{\Omega}
\pp_t \left( |\pp_t y|^2 e^{2s\psi_{\delta,1}}\right)
\,dxdt  \\
&\leq
C
\int_{t_0-\delta}^{t_0}
\int_{\Omega}
\left(
|\pp_t^2 y||\pp_ty| + 
s\va_{\delta,1}^2
|\pp_t y|^2 \right)e^{2s\psi_{\delta,1}}
\,dxdt \\
&\leq
C
\int_{Q_\delta}
s\va_{\delta,1}^2
\left( 
|\pp_t y|^2+|\pp_t z|^2
\right)
e^{2s\psi_{\delta,1}}
\,dxdt.
\end{align*}
Putting this together with \eqref{eq:dfpr09}, we see that
\begin{equation}
\label{eq:dfpr102}
\wJ_2
\leq
\frac{C}{s}
\int_{Q_\delta} 
\va_{\delta,1}^2
\left(
|\nabla \triangle a|^2
+ 
\sum_{|\alpha| \leq 2} |\pp_x^\alpha a|^2
\right)
e^{2s\psi_{\delta,1}}\,dxdt
+C(s) B^2.
\end{equation}
Moreover we have
\begin{align*}
\wJ_3
&=
\int_{t_0-\delta}^{t_0}
\int_{\Omega}
\pp_t \left( |\cA_1 y|^2 e^{2s\psi_{\delta,1}}\right)
\,dxdt  \\
&\leq
C
\int_{t_0-\delta}^{t_0}
\int_{\Omega}
\left(  
| \cA_1\pp_t y||\cA_1 y| + 
s\va_{\delta,1}^2
|\cA_1 y|^2 \right)e^{2s\psi_{\delta,1}}
\,dxdt \\
&\leq
C
\int_{Q_\delta}
s\va_{\delta,1}^2
\left( 
|\cA_1 y|^2+|\cA_1 z|^2
\right)
e^{2s\psi_{\delta,1}}
\,dxdt  \\
&\leq
C
\int_{Q_\delta}
s\va_{\delta,1}^2
\left( 
\sum_{|\alpha| \leq 2}
|\pp_x^\alpha y|^2
+
\sum_{|\alpha|\leq 2}
|\pp_x^\alpha z|^2
\right)
e^{2s\psi_{\delta,1}}
\,dxdt  
\end{align*}
This together with \eqref{eq:dfpr09} gives
\begin{equation}
\label{eq:dfpr103}
\wJ_3
\leq
\frac{C}{s}
\int_{Q_\delta} 
\va_{\delta,1}^2
\left(
|\nabla \triangle a|^2
+ 
\sum_{|\alpha| \leq 2} |\pp_x^\alpha a|^2
\right)
e^{2s\psi_{\delta,1}}\,dxdt
+C(s) B^2.
\end{equation}
We have
\begin{align*}
\wJ_4
&=
\int_{t_0-\delta}^{t_0}
\int_{\Omega}
\pp_t \left( |\nabla y|^2 e^{2s\psi_{\delta,1}}\right)
\,dxdt  \\
&\leq
C
\int_{t_0-\delta}^{t_0}
\int_{\Omega}
\left(
|\nabla  \pp_t y||\nabla  y| + 
s\va_{\delta,1}^2|\nabla y|^2 \right)e^{2s\psi_{\delta,1}}
\,dxdt \\
&\leq
C
\int_{Q_\delta}
s\va_{\delta,1}^2
\left( 
|\nabla y|^2+|\nabla z|^2
\right)
e^{2s\psi_{\delta,1}}
\,dxdt .
\end{align*}
Combining this with \eqref{eq:dfpr09}, we may estimate the right-hand side of the above inequality and we obtain
\begin{equation}
\label{eq:dfpr104}
\wJ_4
\leq
\frac{C}{s^3}
\int_{Q_\delta} 
\va_{\delta,1}^2
\left(
|\nabla \triangle a|^2
+ 
\sum_{|\alpha| \leq 2} |\pp_x^\alpha a|^2
\right)
e^{2s\psi_{\delta,1}}\,dxdt
+C(s) B^2.
\end{equation}
Similarly, we obtain
\begin{align*}
\wJ_5
&=
\int_{t_0-\delta}^{t_0}
\int_{\Omega}
\pp_t \left( |\nabla \pp_t y|^2 e^{2s\psi_{\delta,1}}\right)
\,dxdt  \\
&\leq
C
\int_{t_0-\delta}^{t_0}
\int_{\Omega}
\left(
|\nabla \pp_t^2 y||\nabla \pp_ty| + 
s\va_{\delta,1}^2|\nabla \pp_t y|^2 \right)e^{2s\psi_{\delta,1}}
\,dxdt \\
&\leq
C
\int_{Q_\delta}
s\va_{\delta,1}^2
\left( 
|\nabla \pp_t y|^2+|\nabla \pp_t z|^2
\right)
e^{2s\psi_{\delta,1}}
\,dxdt.
\end{align*}
Putting this together with \eqref{eq:dfpr09}, we see that
\begin{equation}
\label{eq:dfpr105}
\wJ_5
\leq
C
\int_{Q_\delta} 
s\va_{\delta,1}^2
\left(
|\nabla \triangle a|^2
+ 
\sum_{|\alpha| \leq 2} |\pp_x^\alpha a|^2
\right)
e^{2s\psi_{\delta,1}}\,dxdt
+C(s) B^2.
\end{equation}
Moreover we have
\begin{align*}
\wJ_6
&=
\int_{t_0-\delta}^{t_0}
\int_{\Omega}
\pp_t \left( |\nabla \cA_1 y|^2 e^{2s\psi_{\delta,1}}\right)
\,dxdt  \\
&\leq
C
\int_{t_0-\delta}^{t_0}
\int_{\Omega}
\left(  
|\nabla \cA_1\pp_t  y||\nabla \cA_1 y| + 
s\va_{\delta,1}^2|\nabla\cA_1 y|^2 \right)e^{2s\psi_{\delta,1}}
\,dxdt \\
&\leq
C
\int_{Q_\delta}
s\va_{\delta,1}^2
\left( 
|\nabla \cA_1 y|^2+|\nabla \cA_1 z|^2
\right)
e^{2s\psi_{\delta,1}}
\,dxdt  \\
&\leq
C
\int_{Q_\delta}
s\va_{\delta,1}^2
\Bigl(
|\nabla \pp_t y|^2
+
|\nabla (\rho_1\pp_t-\cA_1) y|^2
\\
&\qquad \qquad \qquad \quad
+
|\nabla \pp_t z|^2
+
|\nabla (\rho_1\pp_t-\cA_1) z|^2
\Bigr)
e^{2s\psi_{\delta,1}}
\,dxdt  
\end{align*}
This together with \eqref{eq:dfpr09} gives
\begin{equation}
\label{eq:dfpr106}
\wJ_6
\leq
C
\int_{Q_\delta} 
s\va_{\delta,1}^2
\left(
|\nabla \triangle a|^2
+ 
\sum_{|\alpha| \leq 2} |\pp_x^\alpha a|^2
\right)
e^{2s\psi_{\delta,1}}\,dxdt
+C(s) B^2.
\end{equation}

Summing up the estimate of \eqref{eq:dfpr05} through \eqref{eq:dfpr106}, we have
\begin{align}
\label{eq:dfpr16}
&
\int_{\Omega} \left(|\wF(x,t_0)|^2+|\nabla\wF(x,t_0)|^2\right)e^{2s\psi_{\delta,1}(x,t_0)}\,dx
\\
&\leq
C
\int_{Q_\delta} 
s\va_{\delta,1}^2
\left(
|\nabla \triangle a|^2
+ 
\sum_{|\alpha| \leq 2} |\pp_x^\alpha a|^2
\right)e^{2s\psi_{\delta,1}}\,dxdt
\nonumber \\
&\quad
+
C
\int_{\Omega}  \sum_{|\alpha|\leq 5} |\pp_x^\alpha u(x,t_0)|^2e^{2s\psi_{\delta,1}(x,t_0)}\,dx +C(s) B^2.
\nonumber
\end{align}
Let us estimate the left-hand side of the inequality \eqref{eq:dfpr16} from below. 
By \eqref{eq:dfpr02} at $t=t_0$, we have
\begin{align}
\label{eq:dfpr17}
&a_1(x) \nabla r(x,t_0)\cdot \nabla \triangle a(x) \\
&=
\wF(x,t_0)
-2a_1(x) \sum_{i,j=1}^n (\pp_i \pp_j r(x,t_0)) (\pp_i\pp_j a(x))
\nonumber  \\
&\quad
-a_1(x) \triangle r(x,t_0) \triangle a(x)
-(\nabla a_1(x)-\bb(x))\cdot (\nabla r(x,t_0)\cdot \nabla)\nabla a(x)
\nonumber \\
&\quad
-\Biggl[
(\rho_2 \pp_t^{\frac12} -\rho_1 \pp_t)\nabla r(x,t_0) 
+3 a_1(x)\nabla \triangle r(x,t_0) 
+\left(\triangle r(x,t_0)\right) \nabla a_1(x)
\nonumber \\
&\qquad
-\left(\triangle r(x,t_0)\right) \bb(x)
-c(x) \nabla r(x,t_0)
+\frac{\rho_2\nabla r(x,0)}{\sqrt{\pi t_0}}
\Biggr]\cdot \nabla a(x) 
\nonumber \\
&\quad 
-(\nabla a_1(x)-\bb(x))\cdot (\nabla a(x)\cdot \nabla) \nabla r(x,t_0)
\nonumber \\
&\quad
-\Biggl[
(\rho_2 \pp_t^{\frac12} -\rho_1 \pp_t)\triangle r(x,t_0) 
+\left(\nabla  a_1(x)\cdot\nabla \triangle r(x,t_0)\right) 
+a_1(x)\triangle^2 r(x,t_0) 
\nonumber \\
&\qquad
-\left(\bb(x) \cdot \nabla\triangle r(x,t_0)\right) 
-c(x) \triangle r(x,t_0)
+\frac{\rho_2 \triangle r(x,0)}{\sqrt{\pi t_0}}
\Biggr]a(x),\quad
 x \in \Omega
\nonumber .
\end{align}
Note that 
\begin{equation*}
|\nabla r(x,t_0)\cdot \nabla d_1(x)| \geq m_1>0, \quad x\in \overline{\Omega}, 
\end{equation*}
and $a\in H^4(\Omega)$ satisfies $a(x)=0$, $x\in D$. 
Let us apply the Lemma \ref{lem:ce3rd1} to \eqref{eq:dfpr17} in 
$\Omega$. Then we obtain
\begin{align}
\nonumber 
&
\int_\Omega
\Biggl[
s\va_{\delta,1}(x,t_0)
\sum_{i,j,k=1}^n|\pp_i \pp_j \pp_k a(x)|^2
+
(s\va_{\delta,1}(x,t_0))^{2}
|\nabla \triangle a(x)|^2
\\
&\qquad
+
(s\va_{\delta,1}(x,t_0))^{3}
\sum_{i,j=1}^n|\pp_i \pp_j a(x)|^2
\nonumber \\
&\qquad
+(s\va_{\delta,1}(x,t_0))^{5}
\left(|\nabla a(x)|^2+ |a(x)|^2 \right)
\Biggr]
e^{2s \psi_{\delta,1}(x,t_0)}\,dx
\nonumber \\
&\leq
\int_{\Omega} \left(|\wF(x,t_0)|^2+|\nabla\wF(x,t_0)|^2\right)e^{2s\psi_{\delta,1}(x,t_0)}\,dx 
+
\int_{\Omega} \sum_{|\alpha|\leq 3} |\pp_x^\alpha a(x) |^2 e^{2s\psi_{\delta,1}(x,t_0)}\,dx 
\nonumber 
\end{align}
Taking sufficiently large $s>0$, we may absorb the second term on the right-hand side of the above inequality and we get
\begin{align}
\nonumber 
&
\int_\Omega
\Biggl[
s\va_{\delta,1}(x,t_0)
\sum_{i,j,k=1}^n|\pp_i \pp_j \pp_k a(x)|^2
+
(s\va_{\delta,1}(x,t_0))^{2}
|\nabla \triangle a(x)|^2
\\
&\qquad
+
(s\va_{\delta,1}(x,t_0))^{3}
\sum_{i,j=1}^n|\pp_i \pp_j a(x)|^2
\nonumber \\
&\qquad
+(s\va_{\delta,1}(x,t_0))^{5}
\left(|\nabla a(x)|^2+ |a(x)|^2 \right)
\Biggr]
e^{2s \psi_{\delta,1}(x,t_0)}\,dx
\nonumber \\
&\leq
\int_{\Omega} \left(|\wF(x,t_0)|^2+|\nabla\wF(x,t_0)|^2\right)e^{2s\psi_{\delta,1}(x,t_0)}\,dx 
\nonumber 
\end{align}

Combining this with \eqref{eq:dfpr16}, we obtain
\begin{align}
\label{eq:dfpr19}
&
\int_\Omega
\Biggl[
s\va_{\delta,1}(x,t_0)
\sum_{i,j,k=1}^n|\pp_i \pp_j \pp_k a(x)|^2
+
(s\va_{\delta,1}(x,t_0))^{2}
|\nabla \triangle a(x)|^2
\\
&\qquad
+
(s\va_{\delta,1}(x,t_0))^{3}
\sum_{i,j=1}^n|\pp_i \pp_j a(x)|^2
\nonumber \\
&\qquad
+(s\va_{\delta,1}(x,t_0))^{5}
\left(|\nabla a(x)|^2+ |a(x)|^2 \right)
\Biggr]
e^{2s \psi_{\delta,1}(x,t_0)}\,dx
\nonumber \\
&\leq
C
\int_{Q_\delta} 
s\va_{\delta,1}^2
\left(
|\nabla \triangle a|^2
+ 
\sum_{|\alpha| \leq 2} |\pp_x^\alpha a|^2
\right)e^{2s\psi_{\delta,1}}\,dxdt
\nonumber \\
&\quad
+
C
\int_{\Omega}  \sum_{|\alpha|\leq 5} |\pp_x^\alpha u(x,t_0)|^2e^{2s\psi_{\delta,1}(x,t_0)}\,dx +C(s) B^2.
\nonumber\\
&\leq
C
\int_{\Omega} 
s\va_{\delta,1}^2(x,t_0)
\left(
|\nabla \triangle a(x)|^2
+ 
\sum_{|\alpha| \leq 2} |\pp_x^\alpha a(x)|^2
\right)e^{2s\psi_{\delta,1}(x,t_0)}\,dx
\nonumber \\
&\quad
+
C
\int_{\Omega}  \sum_{|\alpha|\leq 5} |\pp_x^\alpha u(x,t_0)|^2e^{2s\psi_{\delta,1}(x,t_0)}\,dx +C(s)B^2.
\nonumber
\end{align}
In the last inequality, we used the fact that 
\begin{equation*}
\va_{\delta,1}^2(x,t)e^{2s\psi_{\delta,1}(x,t)}\leq
\va_{\delta,1}^2(x,t_0)e^{2s\psi_{\delta,1}(x,t_0)}, \quad (x,t)\in Q_{\delta}
\end{equation*}
for large $s>0$. 

Choose sufficiently large $s>0$ and absorb the first term on the right-hand side of \eqref{eq:dfpr19} 
into the left-hand side. Then we obtain
\begin{align}
\label{eq:dfpr20}
&
\int_\Omega
\Biggl[
s\va_{\delta,1}(x,t_0)
\sum_{i,j,k=1}^n|\pp_i \pp_j \pp_k a(x)|^2
+
(s\va_{\delta,1}(x,t_0))^{2}
|\nabla \triangle a(x)|^2
\\
&\qquad
+
(s\va_{\delta,1}(x,t_0))^{3}
\sum_{i,j=1}^n|\pp_i \pp_j a(x)|^2
\nonumber \\
&\qquad
+(s\va_{\delta,1}(x,t_0))^{5}
\left(|\nabla a(x)|^2+ |a(x)|^2 \right)
\Biggr]
e^{2s \psi_{\delta,1}(x,t_0)}\,dx \nonumber \\
&\leq
C
\int_{\Omega}  \sum_{|\alpha|\leq 5} |\pp_x^\alpha u(x,t_0)|^2e^{2s\psi_{\delta,1}(x,t_0)}\,dx +C(s)B^2.
\nonumber
\end{align}
Fix $s>0$. Noting that $ \va_{\delta,1}(\cdot,t_0)e^{2s\psi_{\delta,1}(\cdot,t_0)}$ has its upper and lower bound in $\overline{\Omega}$, we see that 
\begin{equation*}
\|a \|_{H^3(\Omega)} 
\leq C \| u (\cdot, t_0) \|_{H^5(\Omega)} +CB.
\end{equation*}
Thus we obtain the stability estimate \eqref{df:seb}.  
\end{proof}
\begin{proof}[Proof of Theorem \ref{thm:df2}] 
We may prove Theorem \ref{thm:df2} in the same way as Theorem \ref{thm:df1}. 
\end{proof}

\section*{Acknowledgments}
The authors acknowledge support from the Japan Society for the Promotion of Science (JSPS) A3 foresight program: Modeling and Computation of Applied Inverse Problems. MM also acknowledges support from 
Grant-in-Aid for Scientific Research (17K05572 and 17H02081) of JSPS.

%The author would like to thank the anonymous referees
%and board members for their careful reading, invaluable comments.
%

%
%
%	REFERENCES
%
%

\end{document}